\documentclass{article}
\usepackage{blindtext}
\usepackage[a4paper, total={6in, 8in}]{geometry}

\usepackage{graphicx} 
\usepackage{wrapfig}
\usepackage{amsfonts} 
\usepackage{authblk}
\usepackage{tikz}
\usepackage{hyperref}
\usepackage{orcidlink}
\usepackage{tikz-cd}
\usepackage{listings}
\usepackage{float} 
\usepackage{booktabs} 

\usepackage{amsmath}
\usepackage{multirow}
\usepackage{amsthm}

\usepackage{subcaption} 

\usepackage[all,cmtip]{xy}
\usepackage{blindtext}
\usepackage{titlesec}
\title{Sections and Chapters}

\usepackage[english]{babel}
\newtheorem{theorem}{Theorem}[section]
\newtheorem{prop}{Proposition}[section]
\newtheorem{corollary}{Corollary}[theorem]
\newtheorem{lemma}[theorem]{Lemma}
\newtheorem{definition}{Definition}
\newtheorem{example}{Example}[section]

\newcommand{\C}{\mathbb{C}}
\newcommand{\K}{\mathbb{K}}
\newcommand{\N}{\mathbb{N}}
\newcommand{\mycomment}[1]{}
\begin{document}
\title{Mayer Path Homology}
\author{
\large
Dilan Karaguler\,\orcidlink{0000-0002-4632-337X}$^{1}$\footnote{Corresponding author: \href{mailto:karagule@msu.edu}{karagule@msu.edu}} ,
Guo-Wei Wei\,\orcidlink{0000-0001-8132-5998}$^{1,2,3}$\\

\normalsize
$^1$Department of Mathematics, Michigan State University, MI 48824, USA\\
$^2$Department of Mathematics, University of Georgia, Athens, GA 30602, USA\\
$^3$Department of Biochemistry and Molecular Biology, University of Georgia, Athens, GA 30602, USA
}

\maketitle
\begin{abstract}
We introduce \emph{Mayer path homology}, a new homology theory for directed path complexes obtained by equipping path complexes with an $N$-nilpotent differential. The main novelty of this work is the introduction of an $N$-differential on path complexes, giving rise to $N$-chain complexes of $\partial$-invariant paths and Mayer path homology groups $H_n^{N,q}(P)$.
We prove that this construction defines a canonical invariant of directed graphs and is more sensitive than standard path homology, distinguishing directed network motifs that ordinary path homology cannot separate. We further establish a complete classification of generators of $\Omega_2^N$ and $\Omega_3^N$, determining all admissible combinatorial types. Finally, we characterize elements of the first Mayer path cycles group $Z_1^{N,q}$ in terms of weighted directed cycles arising from spanning-tree constructions.
These results provide the first systematic structural theory for Mayer path complexes and reveal new higher-order algebraic structures in directed graphs.
\end{abstract}

\tableofcontents

\section{Introduction}

In the standard homology setting, the boundary operator satisfies $d^{2} = 0$. 
The homology defined using an $N$-nilpotent boundary map satisfying $d^{N} = 0$ 
for an integer $N \geq 2$ on $N$-chain complexes is called \emph{Mayer homology}. 
This idea was first introduced by Mayer in 1942 \cite{mayer1942new}. 
The notion of an $N$-chain complex naturally arises as a graded object equipped 
with an $N$-nilpotent differential. From this structure, Mayer defined a family 
of homology groups depending on an additional parameter $q$, thereby extending 
the classical construction of cycles and boundaries.

Shortly thereafter, Spanier further developed Mayer’s theory and clarified its 
relationship with classical homology \cite{SpainerMayer}. In particular, he showed 
that Mayer homology coincides with standard homology when the coefficient field 
has prime characteristic $p$. Although Mayer’s original work remained relatively 
isolated for several decades, the underlying idea of $N$-complexes re-emerged in 
modern algebra. In particular, $N$-complexes have been systematically studied as 
natural generalizations of chain complexes, equipped with a differential satisfying 
$d^{N} = 0$, giving rise to a family of homology theories indexed by integers 
$1 \leq q \leq N - 1$. These structures provide a richer algebraic framework in 
which interactions between non-consecutive degrees become relevant.

A major development in this direction is due to Kapranov 
\cite{Kapranovqthanalog}, who introduced a $q$-analog of homological algebra in 
which the usual alternating signs are replaced by powers of a primitive $N$th root 
of unity. In this framework, one can define a Poincar\'e polynomial at $N$th root of unity $\xi$
\[
P_C(\xi) = \sum \dim(C_i)\, \xi^i
\]
for an $N$-complex, which vanishes for exact sequences. Moreover, when $\xi = -1$, 
this polynomial recovers the Euler characteristic in the classical case $N = 2$. 
Later, Tikaradze \cite{TIKARADZE} established a relation between the Poincar\'e 
polynomial and Mayer homology groups, showing that for any $1 \leq q \leq N - 1$ and $\xi$ is the $N$th root of unity,
\[
P_C(\xi) = \frac{1}{1 - \xi^{q}} \sum \xi^{i} 
\big( \dim(H_{i}^{N,q}) - \dim(H_{i-q}^{N, N-q}) \big).
\]

More recently, Mayer homology has found applications in topological data analysis 
and machine learning. In particular, persistent Mayer homology and the associated 
Mayer Laplacians have been introduced as extensions of persistent homology to 
$N$-complexes, providing refined topological and spectral invariants for complex 
data \cite{persistentmayerhomology}. These constructions admit stability 
properties, such as Wasserstein stability of persistence diagrams, and have 
demonstrated effectiveness in capturing multi-scale features in large and 
heterogeneous datasets. Furthermore, persistent Mayer homology has been 
successfully applied in molecular data analysis, where it yields enriched 
topological descriptors for machine learning tasks. Notably, in the context of 
protein--ligand binding affinity prediction, 
Mayer-homology-based features have been shown to outperform classical approaches 
by encoding more detailed geometric and topological information across multiple 
scales \cite{MHLmachinelearning}. This success motivates the in-depth study of Mayer homology.

Path homology,
 introduced by Grigor’yan et al. in \cite{homologiesofpathcomplexesanddigraphs},  is a homology theory for directed graphs that overcomes the limitations of classical simplicial and graph homology. In this framework, path complexes replace simplicial complexes, and chain groups are generated by directed paths, allowing the detection of nontrivial higher-dimensional structures in digraphs. The theory retains key features of classical homology, including chain complexes, homology groups, and functoriality, and was further developed in \cite{GLMY2014,GrigoryanSurvey,homotopytheoryfordigraph}, where analogues of the Eilenberg--Steenrod axioms and additional structural properties were established. The applications of path homology span multiple fields, including molecular and materials science \cite{chen2023path}, network modeling \cite{pathnetworkmodelling,pathnetworkmodelling2,motif}, and graph neural networks \cite{gopaths}. In the context of topological data analysis, persistent path homology \cite{chowdhury2018persistent} and persistent path Laplacian \cite{PersPathLaplacian} have also been developed, providing refined multi-scale topological and spectral invariants. In addition, efficient computational methods have been proposed, including algorithms for computing $1$-dimensional persistent path homology \cite{efficientH1}.

 One of the key limitations of Mayer homology is its dependence on the choice of triangulation or simplicial complex, which may not be canonical for a given dataset or graph. In contrast, path complexes associated to directed graphs offer a canonical and combinatorially intrinsic construction, where the underlying structure is determined directly by the directed paths of the graph. As such, the combination of Mayer homology and path homology  can remove the ambiguity of  Mayer homology,  providing  a canonical invariant of a given directed graph, rather than of an underlying topological space.

By combining these two frameworks, we obtain Mayer path homology, which incorporates the higher-order algebraic flexibility of Mayer homology, given by the $N$-nilpotent differential $d^N = 0$, together with the direction-sensitive and canonical nature of path complexes. This allows us to study higher-order interactions in directed graphs without having to  deal with arbitrary triangulations. 

Furthermore, this combination enhances the ability to capture interdimensional relationships in digraphs. While path homology detects higher-dimensional structures arising from directed paths, Mayer homology provides additional layers of algebraic structure through the parameter $q$, enabling a more refined analysis of how these features interact. 

It offers a unified framework that improves both the structural sensitivity and the interpretability of topological invariants for directed data.

In this work, we investigate the structure of $2$- and $3$-simplices in Mayer path complexes and analyze their corresponding generators. In particular, we provide an explicit description of the generators of these spaces, going beyond existing results that establish only the existence of a basis. For standard path complexes, although there are known theorems guaranteeing the construction of a basis, a complete enumeration of all possible generator types is not available in the literature. 
We address this gap by giving a full classification of generators in low dimensions, which allows for a direct comparison between standard path homology and Mayer path homology. This explicit characterization is essential for verifying structural properties of Mayer path complexes, as the interaction between generators plays a crucial role in understanding the effect of the $N$-nilpotent differential. Consequently, our results provide a more detailed combinatorial and algebraic understanding of path-based homology theories and lay the groundwork for further study of higher-dimensional structures and their persistence.

In Section \ref{Sec:Background}, we briefly review Mayer homology and path homology to  establish notations and facilitate our formulation. Section \ref{Sec:MayerPathHomology} is devoted to the theory of Mayer path homology. This paper ends with a conclusion.

\section{Background}\label{Sec:Background}

\subsection{Mayer Homology}
Throughout this section, we consider a field $\K $ containing a primitive $N$-th root of unity, where $N\geq 2$ is an integer.

\begin{definition}
\begin{enumerate}

    \item  An $N$-chain complex consists of a graded $\K$-linear space $C = (C_n)_{n\geq 0}$, equipped with a linear map $d : C_n \to C_{n-1}$ satisfying $d^N = 0$. The linear map $d_ : C_n\to C_{n-1}$ is called an $N$-differential or $N$-boundary operator.
    \item For an $N$-chain complex $(C_*,d) $ and $1\leq q\leq N-1$, the space of $q$-th $n$-cycles defined as \[Z^{N,q}_{n}=\{x\in C_n\ |\ d^qx=0\}\]
    and the space of the $q$-th $n$-boundaries is defined as \[B^{N,q}_{n}=\{d^{N-q}x\ |\ x\in C_{n+N-q}\}\]
     \end{enumerate}
\end{definition}

\begin{center}

\[
        \begin{array}{cccccccccccc}
        \cdots & \xrightarrow{d} & C_{n+N-1} & \xrightarrow{d^{N-1}} & C_n & \xrightarrow{d}&C_{n-1} & \xrightarrow{d^{N-1}} & \cdots \\
         & & \downarrow{d} & & \downarrow{i} & & \downarrow{d} & & \\

          &  &\vdots &  & \vdots &  & \vdots &  &  \\

        \cdots & \xrightarrow{d^{q}} & C_{n+N-q} & \xrightarrow{d^{N-q}} & C_{n} & \xrightarrow{d^{q}} & C_{n-q} & \xrightarrow{d^{N-q}} & \cdots \\
         & & \downarrow{d} & & \downarrow{i} & & \downarrow{d} & & \\

          &  &\vdots &  & \vdots &  & \vdots &  &  \\

        \cdots & \xrightarrow{d^{N-2}} & C_{n+2} & \xrightarrow{d^2} & C_{n} & \xrightarrow{d^{N-2}} & C_{n-N+2} & \xrightarrow{d^2} & \cdots \\
         & & \downarrow{d} & & \downarrow{i} & & \downarrow{d} & & \\

         \cdots & \xrightarrow{d^{N-1}} & C_{n+1} & \xrightarrow{d} & C_{n} & \xrightarrow{d^{N-1}} & C_{n-N+1} & \xrightarrow{d} & \cdots \\
        \end{array}
        \]
\end{center}

The Mayer homology of the $N$-chain complex $(C_*,d)$ is defined as
\[H^{N,q}_{n}(C_*, d) := Z^{N,q}_{n}/B^{N,q}_{n},\ n \geq  0\]
The Mayer Betti Numbers are the dimension of the Mayer homology groups at corresponding dimensions
\[\beta^{N,q}_{n}=\text{dim}(H^{N,q}_{n}(C_*, d))\]

The following example demonstrates that the Mayer homology groups depend on the chosen triangulation of the topological space. In particular, different triangulations may yield non-isomorphic Mayer homology groups. Consequently, it does not define a topological invariant of spaces, but rather a combinatorial invariant of the chosen simplicial complex.

\begin{figure}[htp]
    \centering
    \subfloat[ $T_1$]{\includegraphics[width=0.3\linewidth]{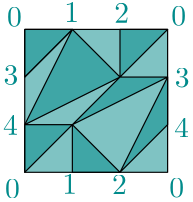}}
    \hfill
    \subfloat[ $T_2$]{ \includegraphics[width=0.3\linewidth]{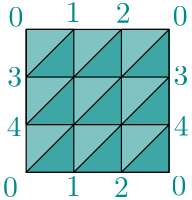}}
    \caption{Examples of triangulations of the torus}
    \label{fig:torusexample}
\end{figure}

\begin{example}
\label{ex:triangulation}
    Let $T_1$ and $T_2$ be two triangulations of the torus where $T_1$ is the minimal triangulation and $T_2$ as in  Figure \ref{fig:torusexample}. 
    The Mayer Homology groups of the simplicial complex for $N=3$ induced by $T_2$ is computed in \cite{persistentmayerhomology}. 
    
\begin{table}[htbp]
\centering
\caption{Mayer Homologies of Torus with respect to different triangulations}
\begin{tabular}{|l|c|c|c|c|c|c|} 
\toprule
 &  $\beta_0^{3,1}$ & $\beta_1^{3,1}$ & $\beta_2^{3,1}$ &$\beta_0^{3,2}$ & $\beta_1^{3,2}$ & $\beta_2^{3,2}$ \\
\hline
$T_1$ & $1$ &$18$  &0  &$0$  & $9$ & $10$ \\
\hline
$T_2$ & $0$ &$14$  &0  &$0$  & $7$ & $7$ \\
\bottomrule
\end{tabular}
\end{table}
\end{example}

For $N=2$, Mayer homology reduces to standard homology, and the Betti number $\beta_0$ counts the number of connected components, where connectivity is defined via edges (i.e., $1$-simplices). For $N \geq 3$, the interpretation of Betti numbers becomes more refined. For each $1 \leq q \leq N-1$, the group $H_0^{N,q}$ measures connected components with respect to a generalized notion of connectivity, where two vertices are considered connected if they are linked through $(N-q)$-dimensional cells. In this sense, the classical edge-based connectivity is replaced by higher-dimensional adjacency relations determined by the $N$-complex structure. A similar interpretation extends to $H_1^{N,q}$, where the corresponding Betti numbers capture higher-order cycle structures formed by interactions among cells of different dimensions, rather than just loops formed by edges.

\subsection{Path Homology}

In this section, we will provide a summary of the work by  Yau and coworkers \cite{Pathcomplexesandtheirhomologies,homologiesofpathcomplexesanddigraphs} as we revisit fundamental concepts in path homology. This includes an overview of paths on a finite set, the boundary operator on the path complex, and the homologies associated with the path complex.

\subsubsection{Paths on Finite Set}

Let $V$ be an arbitrary non-empty finite set which is also called a set of vertices. For $p \in \N$,  an \textit{elementary p-path} on $V$ is any sequence $i_0 \cdots i_p$ of $p + 1$ vertices in $V$. 
The vector space over $\K$ (fixed field) generated by all elementary $p$-paths on $V$ is denoted by $\Lambda_p(V)=\Lambda_p(V,\K)$. Elementary $p$-paths will be denoted as $e_{i_0\cdots i_p}$. For any element $v\in \Lambda_p(V) $, we have a unique representation 
 with $c^{i_0i_1\cdots i_p}\in\K$ as:
\[v=\sum_{i_0,i_1,\cdots, i_p \in V}c^{i_0i_1\cdots i_p}e_{i_0i_1\cdots i_p}\]

\begin{definition}
    
Define $\partial :\Lambda_p\to \Lambda_{p-1}$ so that \[\partial(e_{i_0i_1\cdots i_p})=\sum_{j=0}^p(-1)^j e_{i_0\cdots \hat{i}_j\cdots i_p}\] where $e_{i_0\cdots \hat{i}_j\cdots i_p}$ means that we omit the $j$-th term. 
\end{definition}
It is proven that $(\Lambda_p,\partial)_{p}$ forms a chain complex \cite{Pathcomplexesandtheirhomologies}. 
Since we are dealing with graphs which have no loops, we would like to impose that onto the definition of paths. An elementary path $i_0\cdots i_n$ is called \textit{regular} if any consecutive terms are different i.e $i_k\neq i_{k+1}$ for any $k\in \{0,\cdots n-1\}$. Otherwise, they are called \textit{irregular}.

\begin{definition}
Subspace of $\Lambda_p$ spanned by regular elementary paths 
\[\mathcal{R}_p=\mathcal{R}_p(V):=\textit{span}\{e_{i_0i_1\cdots i_p}\ |\ i_0i_1\cdots i_p \textit{ is regular}\}\]
Complementary subspace spanned by irregular  
\[\mathcal{N}_p=\mathcal{N}_p(V):=\textit{span}\{e_{i_0i_1\cdots i_p}\ |\ i_0i_1\cdots i_p \textit{ is not regular}\}\]
\end{definition}

Note that $\partial (e_{121})=e_{21}-e_{11}+e_{12}$ which implies $\partial\mathcal{R}_p\not\subset \mathcal{R}_{p-1}$. 
Since the boundary of irregular term has to contain irregular terms, we have $\partial\mathcal{N}_p\subset \mathcal{N}_{p-1}$. 
We have $\mathcal{N}_p\cap \mathcal{R}_p= \{0\}$ thus $\Lambda_p=\mathcal{R}_p \bigoplus \mathcal{N}_p$. So by using isomorphism between chain complexes $\mathcal{R}_p \simeq \Lambda_p/ \mathcal{N}_p$, we can induce a boundary map on $\mathcal{R}_p$.
The induced boundary map $\partial^{reg}:\mathcal{R}_p\to \mathcal{R}_{p-1}$ is called a regular boundary map. This map assigns $0$ to irregular paths in the image. So, the chain complex of $V$ can be denoted as $(\mathcal{R}_p,\partial^{reg})_{p}$. 

The induced boundary operator is referred to as the regular boundary operator, while the original boundary map is called the non-regular boundary operator in \cite{homologiesofpathcomplexesanddigraphs}. Throughout the rest of this paper, unless otherwise stated, $\partial$ denotes a non-regular boundary operator.

\subsubsection{Path Complex}
\begin{definition}
    A path complex over a set V is a nonempty collection $P$ of elementary paths on V for any $n\in\N$ with the property:
    \begin{center}
         if $i_0\cdots i_n\in P$, then $i_0\cdots i_{n-1}\in P$ and $i_1\cdots i_n \in P$. 
    \end{center}

\end{definition}
The elementary paths in $P$ are called \textit{allowed} while other elementary paths on $V$ that are not in $P$ are called \textit{non-allowed}.

A path complex $P$ is called regular if it does not contain any $1$-path of the form $ii$. $P$ is called strictly regular if it is regular and contains
no path of the form $iji$.

\begin{definition}
    A digraph $G$ is a pair $(V,E)$ where $V$ is the set of vertices and $E\subset V\times V$ is the set of edges. An edge is called directed $(\gamma_1,\gamma_2)\in E$ and the direction will be $\gamma_1\rightarrow \gamma_2$
\end{definition}

\begin{example}
    \begin{itemize}
        \item An abstract finite simplicial complex $K$ is a collection of subsets of a finite vertex set V that satisfies the following property if $\sigma \in K$, then any subset of $\sigma$ is also in $K$. After fixing an ordering of the vertices, $\sigma$ can be transformed to an elementary path over $V$. Denote the new representation as $P(K)$. So the allowed $n$-paths in $P(K)$ are $n$-simplexes. A path complex can be derived from a finite simplicial complex.

        \item A path complex can be derived from digraph G by taking all paths in G as allowed paths. 
       
    \end{itemize}
\end{example}
 The path complex $P$ of a digraph is
strictly regular if and only if the digraph has no $1$-path $ii\in P$ and no two-way edges $ij,ji\in P$.

\subsubsection{Path Homology}
For a given path complex $P$, we define the $\K$-linear space $\mathcal{A}_n$ as a span of all elementary $n-$path from $P$.
\[\mathcal{A}_n=\mathcal{A}_n(P)=\textit{span}\{e_{i_0i_1\cdots i_n}\ |\ i_0i_1\cdots i_n \in P\}\]

  $(\mathcal{A}_n,\partial)$ may not form chain complex all the time for non-regular $\partial$. For example,
      for $P=\{e_{0,1},e_{12},e_{0,1,2}\}$, we have $\partial(e_{0,1,2})=e_{12}-e_{0,2}+e_{0,1}$ where $e_{0,2}$ is not allowed.
  
  We are changing our domain to a restricted version $\Omega_{n}$ which is defined as infimum chain complexes  \[\Omega_{n}=\Omega_{n}(P)=\{v\in \mathcal{A}_n\ |\ \partial v\in\mathcal{A}_{n-1}\}.\]
 Observe that $\partial\Omega_n\subset \Omega_{n-1}$. The elements of $\Omega_{n}$ are called $\partial$-invariant $n$-paths. The homology of $(\Omega,\partial)$ will be called the path homology groups of the path complex P 

\[H_n=H_n(P)={\rm Ker}\partial|_{\Omega_n}/ {\rm Im}\partial|_{\Omega_{n+1}}.\] 

The regular boundary operator is used to form regular $\partial^{reg}$-invariant $n$-paths. 
\[\Omega^{reg}_{n}=\Omega^{reg}_{n}(P)=\{v\in \mathcal{A}_n\ |\ \partial^{reg} v\in\mathcal{A}_{n-1}\}.\]

In a similar convension, regular path homology groups of a path complex is 
\[H_n^{reg}=H^{reg}_n(P)={\rm Ker}\partial^{reg}|_{\Omega_n}/ {\rm Im}\partial^{reg}|_{\Omega_{n+1}}.\]

The difference between regular and non-regular boundary maps induces homology definitions where non-regular recognize $e_{10}+e_{0,1}$ as a hole and regular does not. 
In \cite{homologiesofpathcomplexesanddigraphs}, the condition for regular and non-regular homologies to be same is stated as \[ \text{If } P \text{ is strictly regular, then } \Omega_n=\Omega^{reg}_n\ \forall n\geq -1 \]

\section{Mayer Path Homology}\label{Sec:MayerPathHomology}

In this section, we introduce Mayer Path Homology. 
For any $ n \in \mathbb{Z}_{\ge 0} $, let $\Lambda_p = \Lambda_p(V)$ be the $\mathbb{K}$-linear space spanned by all elementary $p$-paths with coefficients in the field $\mathbb{K}$. 
Throughout this paper, we take $\mathbb{K} = \mathbb{C}$ and let $\xi $ be primitive $N$-th root of unity. 
Elementary $p$-paths are denoted by $e_{i_0 \cdots i_p}$. 

\begin{definition}
Define the boundary operator $\partial : \Lambda_p \to \Lambda_{p-1}$ by
\[
\partial(e_{i_0 i_1 \cdots i_p}) = \sum_{j=0}^p \xi^j e_{i_0 \cdots \hat{i_j} \cdots i_p},
\]
where $\xi$ is an $N$-th primitive root of unity and $e_{i_0 \cdots \hat{i_j} \cdots i_p}$ denotes the path obtained by omitting the $j$-th index.
\end{definition}

In \cite{Kapranovqthanalog}, the $q$-analogues of the basic numbers and basic factorials for any $q\in\C$ are defined by
\[
[n]_q=\frac{(1-q)^n}{1-q}=1+q+\cdots+q^{n-1},
\]
and
\[
[n!]_q=[1]_q[2]_q\cdots[n]_q
=\sum_{w\in S_n}q^{\ell(w)},
\]
where $\ell(w)$ denotes the length of the permutation $w\in S_n$.
 By \cite[Lemma 0.3]{Kapranovqthanalog}, for each $r\geq 1$ one has
\[
\partial^r
=
[r!]_{q}
\sum_{1\leq j_1<\cdots<j_r}
\xi^{j_1+\cdots+j_r}
\partial_{j_1}\cdots\partial_{j_r},
\]
where $\partial_i$ is the operator obtained by deleting the $i$-th index and $\partial = \sum_{i}^nq^i\partial_i$.
\begin{lemma}
$(\Lambda_*, \partial)$ forms an $N$-chain complex, where $\Lambda_* = \{\Lambda_p\}_{p \in \mathbb{N}}$.
\end{lemma}
\begin{proof}
Taking $r=N$ and $q=\xi$ in the above formula gives where $\xi $ is the $N$th root of unity
\[
\partial^N
=
[N!]_{\xi}
\sum_{1\leq j_1<\cdots<j_N}
\xi^{j_1+\cdots+j_N}
\partial_{j_1}\cdots\partial_{j_N}.
\]
Since $\xi$ is a primitive $N$-th root of unity, we have
\[
[N]_{\xi}=1+\xi+\cdots+\xi^{N-1}=0.
\]
Therefore,
\[
[N!]_{\xi}
=
[1]_{\xi}[2]_{\xi}\cdots[N]_{\xi}
=0.
\]
Hence
\[
\partial^N=0.
\]
\end{proof}

The following example shows that regular boundary operator does not necessarily form $N$-chain complex.
\begin{example}
    Let $G$ be the digraph where $V=\{e_0,e_1\}$ and $E=\{e_{01},e_{10}\}$. Let $N=3$ and $\xi $ is the primitive root of unity

    \begin{equation}
  \begin{split}
    (\partial^{reg})^3(e_{0101})&=(\partial^{reg})^2(e_{101}+e_{010})\\
    &=\partial^{reg}((1+\xi^2)(e_{01}+e_{10}))\\
    &=(1+\xi^2)(1+\xi)(e_0+e_1)=e_0+e_1
  \end{split}
\end{equation}
   
\end{example}
Using the regular boundary operator removes the irregular terms that contribute to the coefficient $[N!]_{\xi}$. Consequently, the cancellation required for $\partial^N=0$ no longer occurs. For this reason, throughout the Mayer setting, we will always use a non-regular boundary map.

For a given path complex $P$, define the $\mathbb{C}$-linear space $\mathcal{A}_n$ as the span of all elementary $n$-paths from $P$:
\[
\mathcal{A}_n = \mathcal{A}_n(P) = \operatorname{span}\{ e_{i_0 i_1 \cdots i_n} \mid i_0 i_1 \cdots i_n \in P \}.
\]
For example,
\[
\partial(e_{123}) = e_{23} + \xi e_{13} + \xi^2 e_{12},
\]
where $\xi$ is the $N$-th root of unity and $e_{13}\not\in\mathcal{A}_1$. 
Hence, $(\mathcal{A}_n, \partial)$ does not always form an $N$-chain complex.

\begin{definition}
    For $N\geq 2$ and $1\leq q\leq N-1$, define the space of $\partial$-invariant $n$-paths at level $(N,q)$ as
\[
\Omega_n^{N,q} = \Omega_n^{N,q}(P) = \{ v \in \mathcal{A}_n \mid \partial^q v \in \mathcal{A}_{n-q} \}.
\]

The intersection is called the space of $\partial$-invariant $n$-paths and is denoted as follows:
\[
\Omega_n^N = \bigcap_{1 \le q \le N-1} \Omega_n^{N,q}(P)
= \{ v \in \mathcal{A}_n \mid \partial^q v \in \mathcal{A}_{n-q} \}.
\]
\end{definition}

If $v \in \Omega_n^N$, then $v \in \Omega_n^{N,q}$ for all $q \in \{1, \ldots, N-1\}$. 
Thus, for $1\leq q\leq N-2$, $\partial^q v \in \mathcal{A}_{n-q}$. 
Consequently, $\partial^{q-1}(\partial v) \in \mathcal{A}_{n-q}$ for all $q \in \{1, \ldots, N-1\}$, which implies
\[
\partial v \in \Omega_{n-1}^{N,q} \quad \text{for all } q \in \{1, \ldots, N-2\}.
\]
For $q = N-1$, we have $\partial^{N-1}(\partial v) = 0 \in \mathcal{A}_{n-N}$. 
Thus, $\partial(v) \in \Omega_{n-1}^{N,q}$ for all $1 \le q \le N-1$, and therefore
$
(\Omega_n^N, \partial)
$
forms an $N$-chain path complex. The following example shows that, in general, $\Omega_n^{N,q}$ alone does not form a chain complex.
\begin{figure}
    \centering
    \includegraphics[width=0.3\linewidth]{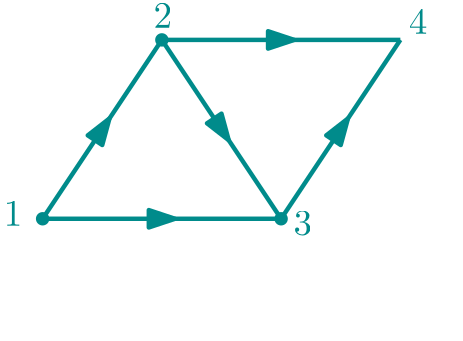}
    \caption{Example 3.1}
    \label{fig:3.1example}
\end{figure}

\begin{example}\label{diamond_example}

    Let $P$ be the path complex induced by the digraph with $V=\{e_1,e_2,e_3,e_4\}$ and $E=\{e_{12},e_{13},e_{23},e_{24},e_{34}\}$. 
    The elementary path spaces will be as follows 
    
    \[\mathcal{A}_2=<e_{123},e_{124},e_{13,4},e_{234}>, \quad   \mathcal{A}_3=<e_{1234}>.\]
    We compute
    \begin{align*}
       \partial(e_{1234})&=e_{234}+\xi e_{13,4}+\xi^2 e_{124}+\xi^3 e_{123}\\
    \partial^2(e_{1234})&=(\xi^2+\xi^3)e_{14}+w, \ \ \ w\in\mathcal{A}_1
    \end{align*}

    For $(\xi^2+\xi^3)e_{14}\in\mathcal{A}_1$, we have $\xi^2+\xi^3=0$ which occurs only when $N=2$. 
    We have the $\partial$-invariant paths as follows

 \begin{align*}
    \Omega^{2,1}_2 &= \Omega^{2}_2 = \langle e_{123}, e_{124}-e_{13,4}, e_{234}\rangle, & 
    \Omega^{2,1}_3 &= \Omega^{2}_3 = \langle e_{1234}\rangle, \\
    \Omega^{3,1}_2 &= \langle e_{123}, e_{124}-e_{13,4}, e_{234}\rangle, & 
    \Omega^{3,2}_2 &= \langle e_{123}, e_{124}, e_{13,4}, e_{234}\rangle, \\
    \Omega^{3,1}_3 &= \langle e_{1234}\rangle, & 
    \Omega^{3,2}_3 &= 0, \\
    \Omega^{3}_2\   &= \Omega^{3,1}_2 \cap \Omega^{3,2}_2 = \langle e_{123}, e_{124}-e_{13,4}, e_{234}\rangle, & 
    \Omega^{3}_3\   &= \Omega^{3,1}_3 \cap \Omega^{3,2}_3 = 0.
\end{align*}
    Observe that $\partial(\Omega^{3,1}_3)\not\subset \Omega^{3,1}_2$ where $\partial(\Omega^{2}_3)\subset \Omega^{2}_2$.
    
\end{example}

\begin{definition}
Let $P$ be a path complex associated with the $N$-chain complex $(\Omega_n^N, \partial)$. 
For any $1 \le q \le N-1$, define the following spaces:

\begin{itemize}
    \item The space of $q$-th $n$-cycles:
    \[
    Z_n^{N,q} = \{ x \in \Omega_n^N \mid \partial^q(x) = 0 \}.
    \]

    \item The space of $q$-th $n$-boundaries:
    \[
    B_n^{N,q} = \{ \partial^{N-q}(x) \mid x \in \Omega_{n+N-q} \}.
    \]
\end{itemize}

Then, the Mayer Path Homology is defined as
\[
H_n^{N,q}(P) := Z_n^{N,q} / B_n^{N,q} 
= \ker(\partial_n^q) / \operatorname{im}(\partial_{n+N-q}^{N-q}).
\]

The rank of $H_n^{N,q}(P)$ is called the Mayer path Betti number.
\end{definition}

For $q> n$, $\partial^q(v)=0$ or in $\mathcal{A}_0$ for every $v\in \Omega_n^N$. Consequently, $Z_n^{N,q}=\Omega_n^N$

\begin{example} Let $P$ be the path complex described in Example 3.1 and in Figure \ref{fig:3.1example}
    
    For $N=2$,
\[0\rightarrow\Omega^{2}_3\rightarrow\Omega^{2}_2\rightarrow\Omega^{2}_1\rightarrow\Omega^{2}_0\rightarrow 0 \]

\begin{align*}
    Z_0^{2,1} &= \Omega^{2}_0, & 
    Z_1^{2,1} &= \langle -e_{12}+e_{13}-e_{23}, -e_{23}+e_{24}-e_{34} \rangle, \\
    Z_2^{2,1} &= \langle e_{123}-e_{124}+e_{13,4}-e_{234} \rangle, & 
    Z_3^{2,1} &= 0, \\
    B_0^{2,1} &= \langle e_2-e_1, e_3-e_1, e_4-e_2 \rangle, & 
    B_1^{2,1} &= \langle e_{12}-e_{13}+e_{23}, e_{34}-e_{24}+e_{23} \rangle, \\
    B_2^{2,1} &= \langle e_{123}-e_{124}+e_{134}+e_{234} \rangle, & 
    B_3^{2,1} &= 0, \\
    H_0^{2,1} &= \Omega^{2}_0 / B_0^{2,1} = \mathbb{C}, & 
    H_1^{2,1} & =H_2^{2,1}=H_3^{2,1} =0, 
\end{align*}

\[H_n^{2,1}(P)=\begin{cases}
    \C & n=0\\
    0 & n\geq 1.
\end{cases}\]

 For $N=3$ and $q=1$,

 \[0\rightarrow\Omega^{3}_2=<e_{123},e_{124}-e_{134},e_{234}>\rightarrow\Omega^{3}_1\rightarrow\Omega^{3}_0\rightarrow 0, \]

\begin{align*}
    Z_0^{3,1}&=\Omega^{3}_0,\quad Z_1^{3,1}=<-\xi e_{12}+\xi e_{13}+e_{24}-e_{34}>,\quad Z_2^{3,1}=0,\\
    B_0^{3,1}&=<-\xi e_1-e_2 +\xi^2 e_3,-e_2+e_3,- e_2-\xi e_3-\xi e_4>,\quad  \ B_k^{3,1}=0\ \forall k\geq 1\\
    H_0^{3,1}&= H_1^{3,1}=\C,\quad  H_2^2=0,
\end{align*}

\[H_n^{3,1}(P)=\begin{cases}
    \C & n=0,1\\
    0 & n\geq 2.
\end{cases}\]

For $N=3$ and $q=2$,
\begin{align*}
Z_0^{3,2}&=\Omega^{3}_0, \quad Z_1^{3,2}=\Omega^{3}_1,\quad Z_2^{3,2}=0,\\
B_0^{3,2}&=<e_2+\xi e_1,e_3+\xi e_1,e_3+\xi e_2,e_4+\xi e_2>,\\
B_1^{3,2}&=<\xi^2e_{12}+\xi e_{13}+e_{23},-e_{34}+e_{24}-\xi^2 e_{13}+\xi^2 e_{12},e_{34}+\xi e_{24}+\xi^2 e_{23}>,\\
H_0^{3,2}&=Z_0^{3,2}/ B_0^{3,2}=0,\ \ H_1^{3,2}=Z_1^{3,2}/B_1^{3,2}=\C^2,\ \ H_2^2=Z_2^{3,2}=0,
\end{align*}

\[H_n^{3,2}(P)=\begin{cases}
    \C^2 & n=1\\
    0 & n=0\text{ and }n\geq 2
\end{cases}\]

\end{example}

In contrast to the classical simplicial setting, where different triangulations of the same topological space may lead to different Mayer homology groups as shown in the Example \ref{ex:triangulation} previously. The path-complex construction associated to a directed graph is canonical. Consequently, Mayer Path Homology does not suffer from triangulation ambiguity and yields a well-defined invariant of directed graphs up to isomorphism.

In the following example, we consider fundamental directed network motifs, including the biparallel, bi-fan, and the $4$-node feedback loop, which frequently arise in complex systems such as neural networks and electronic circuits \cite{networkmotifs}.

\begin{figure}

 \centering
\subfloat[$L_1$]{\includegraphics[width=0.25\linewidth]{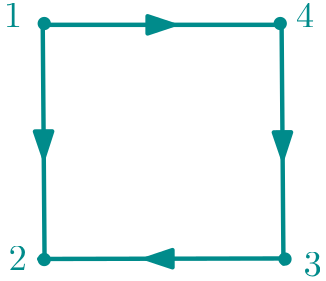}}\label{L1}
\hfill
\subfloat[$L_2$]{\includegraphics[width=0.25\linewidth]{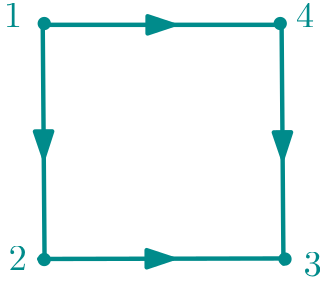}}\label{L2}
\hfill
\subfloat[$L_3$]{\includegraphics[width=0.25\linewidth]{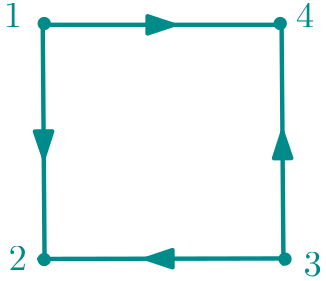}}\label{L3}
\caption{The fundamental network motifs; (A) 4 node-feedback loop, (B) bi-parallel and (C) bi-fan used in Example \ref{nonisomorhiccomparison}  }
\label{fig:three graphs}
\end{figure}
 
\begin{example}
\label{nonisomorhiccomparison}

    Let $L_1,L_2$ and $L_3$ be the digraphs as in Figure \ref{fig:three graphs} sharing the common vertex set $V$ and $E_i$ be the edge set for each. For each case, $\Omega_0^N(L_i)=Z_0^{2,1}(L_i)=Z_0^{3,1}(L_i)=Z_0^{3,2}(L_i)=<V>$, $\Omega_1^N(L_i)=Z_1^{3,2}(L_i)=<E_i>$ where $i=1,2,3$. Observe that $\Omega_2^{N}(L_1)=\Omega_2^{N}(L_3)=\emptyset$ where $\Omega_2^{N}(L_2)=\langle e_{14,3}-e_{{1,2,3}}\rangle$
    \[ B_1(L_1)= \bordermatrix{%
   & e_{12} & e_{14} & e_{32} & e_{43}\cr
e_{1} & \xi & \xi & 0 & 0 \cr
e_{2} & 1 & 0 &  1 & 0 \cr
e_{3} &0 & 0 &  \xi & 1\cr
e_{4} &0 & 1 &  0 &\xi\cr
     },\quad B_1(L_3)= \bordermatrix{%
   & e_{12} & e_{14} & e_{32} & e_{34}\cr
e_{1} & \xi & \xi & 0 & 0 \cr
e_{2} & 1 & 0 &  1 & 0 \cr
e_{3} &0 & 0 &  \xi & \xi\cr
e_{4} &0 & 1 &  0 &1\cr
     } \] \[B_1(L_2)= \bordermatrix{%
   & e_{12} & e_{14} & e_{23} & e_{43}\cr
e_{1} & \xi & \xi & 0 & 0 \cr
e_{2} & 1 & 0 &  \xi & 0 \cr
e_{3} &0 & 0 &  1 & 1\cr
e_{4} &0 & 1 &  0 &\xi\cr
     },\quad B_2(L_2)= \bordermatrix{%
   & e_{143}-e_{123}  \cr
e_{12} & -\xi^2   \cr
e_{14} &\xi^2   \cr
e_{23} & -1   \cr
e_{43} &1   \cr
     }\]
     \[\quad B_1B_2(L_2)=\bordermatrix{%
   & e_{143}-e_{123}  \cr
e_{1} &  0  \cr
e_{2} &-\xi^2-\xi   \cr
e_{3} & 0   \cr
e_{4} &  \xi^2+\xi\cr
     }.\]
     
     For $N=2$, $\xi=-1$
     \begin{align*}
         Z_1^{2,1}(L_1)&=Ker(B_1(L_1))=\langle e_{14}+e_{43}+e_{32}-e_{12} \rangle, \\ 
         Z_1^{2,1}(L_2)&=Ker(B_1(L_2))=\langle e_{14}+e_{43}-e_{23}-e_{12} \rangle,\\
         Z_1^{2,1}(L_3)&=Ker(B_1(L_3))=\langle e_{14}-e_{34}+e_{32}-e_{12} \rangle,\\
         B_0^{2,1}(L_1)&=B_0^{2,1}(L_2)=B_0^{2,1}(L_3) = \langle e_2-e_1, e_3-e_4,e_2-e_3 \rangle
     \end{align*}

     For $N=3$, $\xi=e^{\frac{2\pi i}{3}}$
    \begin{align*}
        Z_1^{3,1}(L_1)&=Ker(B_1(L_1))=0\quad Z_1^{3,1}(L_2)=Ker(B_1(L_2))=\langle \xi e_{12}-\xi e_{14}-e_{23}+e_{43}\rangle,\\
        Z_1^{3,1}(L_3)&=Ker(B_1(L_3))= \langle e_{12}-e_{14}-e_{32}+e_{34}\rangle,\\
        B_0^{31}(L_1)&=0 \quad B_0^{31}(L_2)=\langle(\xi+\xi^2)(e_4-e_2)\rangle,\quad B_0^{3,1}(L_3)=0\\
        B_0^{3,2}(L_1)&=\langle e_4+\xi e_1,e_3+\xi e_4, e_2+\xi e_3,e_2+\xi e_1\rangle, \\
        B_0^{3,2}(L_2)&=\langle  e_4+\xi e_3, e_2+\xi e_3,e_2+\xi e_1\rangle,\quad B_0^{3,2}(L_3)=\langle  e_3+\xi e_4, e_3+\xi e_2,e_2+\xi e_1\rangle, \\
         B_1^{3,2}(L_2)&=\langle-\xi^2e_{12}+\xi^2e_{14}-e_{23}+e_{43} \rangle.
    \end{align*}

Observe that $B_1^{3,1}(L_i)=0$ for all $i=1,2,3$ and $B_1^{3,2}(L_i)=0$ for $i=1,3$.

\begin{table}[htbp]
\centering
\caption{Comparison of path and Mayer path homologies for $L_1,L_2,L_3$}
\renewcommand{\arraystretch}{1.3}
\begin{tabular}{|c|c|c|c|c|c|c|}
\hline
 & \multicolumn{2}{|c|}{$N=2$} & \multicolumn{4}{|c|}{$N=3$} \\\hline
\cline{2-7}
 & \multicolumn{2}{|c|}{$q=1$} & \multicolumn{2}{|c|}{$q=1$} & \multicolumn{2}{|c|}{$q=2$} \\\hline
\cline{2-7}
 & $H_0^{2,1}$ & $H_1^{2,1}$ & $H_0^{3,1}$ & $H_1^{3,1}$ & $H_0^{3,2}$ & $H_1^{3,2}$ \\\hline

$L_1$ & $\C$ & $\C$ & $\C^4$ & $0$ & $0$ & $\C^4$ \\\hline

$L_2$ & $\C$ & $0$ & $\C^3$ & $\C$ & $\C$ & $\C^3$ \\\hline

$L_3$ & $\C$ & $\C$ & $\C^4$ & $\C$ & $\C$ & $\C^4$ \\

\hline
\end{tabular}
\end{table}
\end{example}

It is known that standard path homology can distinguish certain motifs, such as the biparallel and the bi-fan structures \cite{chowdhury2018persistent}. However, it fails to differentiate between others; in particular, the $4$-node feedback loop and the bi-fan yield identical path homology groups. 
In contrast, Mayer path homology provides a finer invariant when $N=3$, it assigns distinct Betti numbers to each of these motifs, thereby distinguishing structures that are indistinguishable under standard path homology. This highlights the increased sensitivity of Mayer path homology in detecting higher-order structural differences in directed networks.

\subsection{Properties of $N$-Chain Complexes of Path Complexes induced by Digraphs}

The properties of allowed and $\partial$-invariant paths on digraphs are examined in detail by Grigor’yan et al \cite{homotopytheoryfordigraph}. 
After modifying the definition of the boundary map, we will investigate whether the properties remain unchanged or undergo alterations. Before doing so, we recall some notions. From now on, by a digraph we mean a simple directed graph. Simple directed graph is a directed graph with no multiple edges and no loops.
\subsubsection{$\partial$-invariant $2$-path $\Omega_2^{N}$}

\begin{figure}
 \centering

\subfloat[ ]{\includegraphics[width=0.3\linewidth]{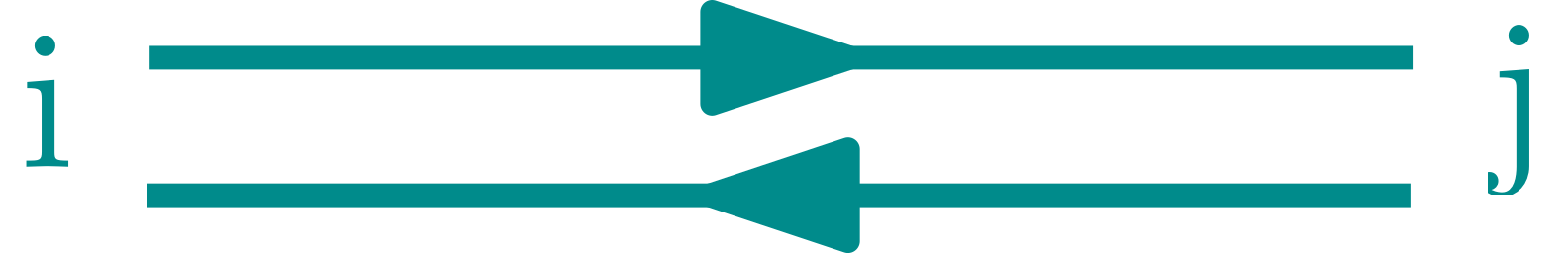}}
\hfill
\subfloat[]{\includegraphics[width=0.3\linewidth]{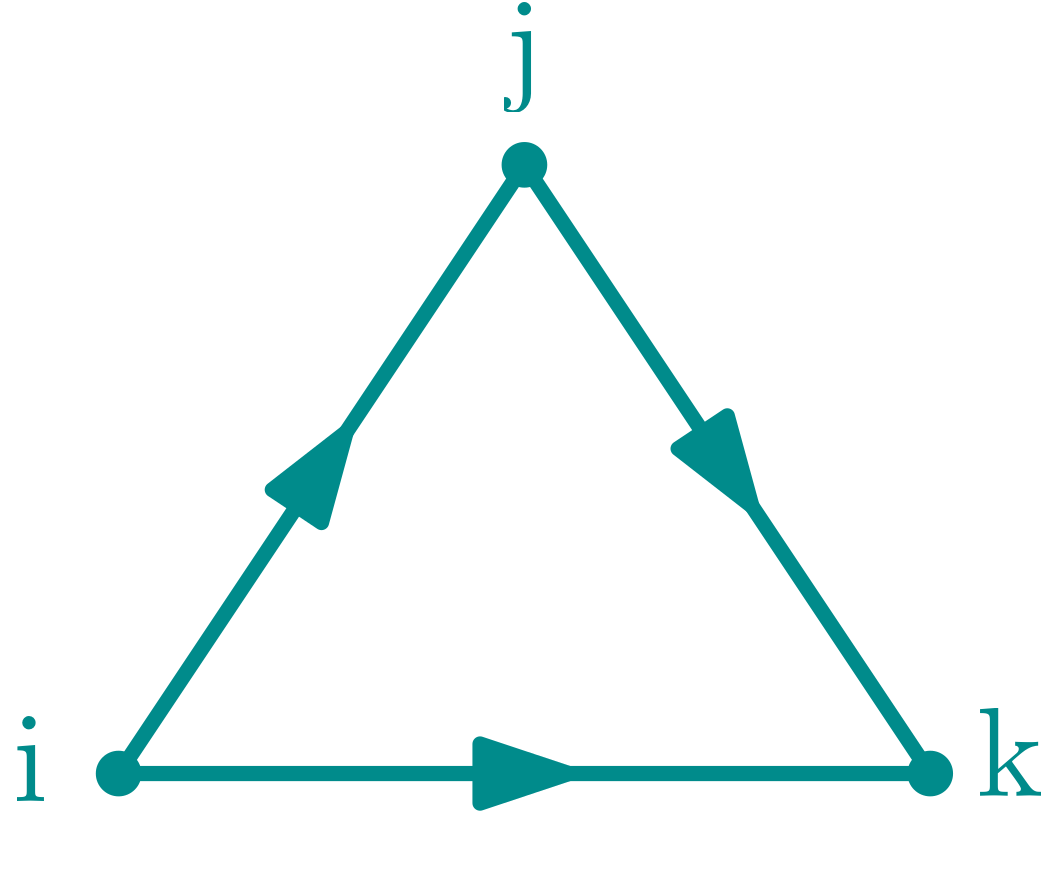}}
\hfill
\subfloat[]{\includegraphics[width=0.3\linewidth]{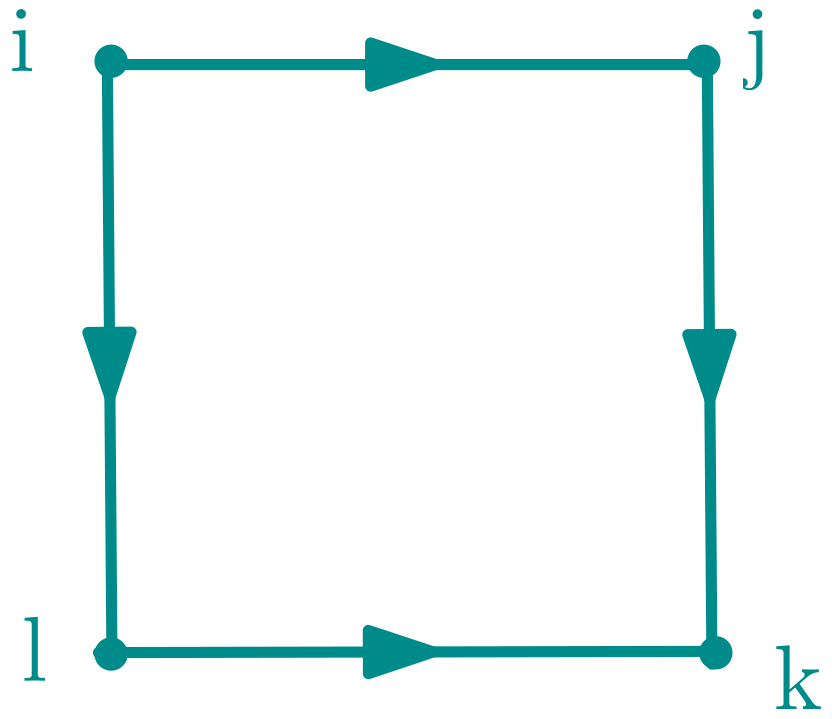}}
\caption{The generators of $\Omega_2^N$: (A) Double edge, (B) Triangle and (C) Square  }
\end{figure}

We have three types of edge compositions in digraph $G$ as 
\begin{itemize}
    \item $e_{iji}$ a double edge in $G$
    \item $e_{i,j,k}+e_{i,k}$ a triangle in $G$
    \item $e_{i,j,k}+e_{i,l,k}$ a square in $G$
\end{itemize}

  It was proven in \cite{homotopytheoryfordigraph} that any element of $\Omega^2_2$ of path complex induced by finite digraph $G$ can be represented as a linear combination of double edge, triangle and square. There was no classification for higher chain structure in \cite{homotopytheoryfordigraph}. We will show that the structure of $\Omega^N_2$ is the same as $\Omega^2_2$.

  \begin{theorem}
  \label{thm:omega_2}
      Elements of $\Omega_2^N$ of path complex $P(G)$ induced by finite digraph $G$ can be represented as a linear combination of double edge, triangle and square for $N\geq2$. 
  \end{theorem}

  \begin{proof}
      Let $v\in\Omega^N_2$, then it is a $\C$-linear combination of elementary paths of the form $e_{i,j,k}$. If $i=k$, we have double edge. If $i\neq k$ and $e_{i,k }\in \mathcal{A}_1(P(G))$, then it will form a triangle.

      Now assume that $e_{i,k}\not\in \mathcal{A}_1(P(G))$. We have $\partial(e_{i,j,k})=e_{i,j}+\xi e_{i,k}+\xi^2 e_{j,k}$ where $\xi $ is $N$th root of unity.

    Let us investigate which scenarios will result in the appearance of $e_{i,k}$ in the image of the boundary map. There are three options; $e_{i,k,m},e_{i,m,k},e_{m,i,k}$ for some $m$. Since $e_{i,k}\not\in \mathcal{A}_1(P(G))$, the only possibility is $e_{i,m,k}$

    We can cancel $\xi e_{i,k}$ by two methods. The first is by negation version $-\xi e_{i,k}$ which means $-e_{i,m,k} $ must appear in $v$ for some $m$. This is the same case explained in \cite[Proposition 2.9]{homotopytheoryfordigraph} that results in a square.

    The other method which is unique to Mayer boundary map is canceling by using the property $1+\xi+\cdots +\xi^{N-1}=0$. Let $v= e_{i,j,k}+\sum_{m=1}^{N-1}\xi^m e_{i,j_m,k}$ where $j_m\neq j$ for at least one $m\in\{1,\cdots,N-1\}$ be called Mayer-type square form.
    \begin{align*}
        \partial(e_{i,j,k}+\sum_{j=1}^{N-1}\xi^{j} e_{i,j_m,k})&=\partial(e_{i,j,k})+\sum_{j=1}^{N-1}\xi^{j} \partial(e_{i,j_m,k})\\
        &=\sum_{j=0}^{N-1}\xi^{j} e_{i,k}+w=w\ \ \ , w\in \mathcal{A}_1(P(G))
    \end{align*}

     There is no need to check $\partial^q$, where $1\leq q\leq N-1$ since the image is in either $\mathcal{A}_0(P(G))$ or $0$. Observe that $e_{i,j,k}-e_{i,j_m,k}$ is also a square element for each $j_m\neq j$ and $\sum_{j_n=j}\xi^ne_{i,j,k}$ is missing $\xi^m$ coefficients where $j_m\neq k$, thus by using $\sum_{i=0}^{N-1}\xi^i=0$ the sum can be rewritten as $\sum_{j_m\neq j}-\xi^me_{i,j,k}$. The element $v$ can be represented as follows
    
     \begin{align*}
         v&=\sum_{j_n=j}\xi^ne_{i,j,k}+\sum_{j_m\neq j}\xi^me_{i,j_m,k}\\
         &=\sum_{j_m\neq j}-\xi^me_{i,j,k}+\sum_{j_m\neq j}\xi^me_{i,j_m,k}=\sum_{j_m\neq j}-\xi^m(e_{i,j,k}-e_{i,j_m,k})
     \end{align*}

 $e_{i,j,k}+\sum_{j=1}^{N-1}\xi^{j} e_{i,j_m,k}$ is a linear combination of the squares $e_{i,j,k}-e_{i,j_m,k}$. The new cancellation method does not produce an independent generator. 
 \end{proof}

\subsubsection{$\partial $-invariant $3$-path $\Omega_3^{N}$}

Grigoryan constructed a basis for $\Omega_3^2$ in \cite{GrigoryanSurvey} under certain conditions such as free of double edge and multisquares. Later Li-Shen removed these assumptions and provided a general construction of a basis \cite{constructionomega3basis}. In this section we will explore whether these generator constructions are valid in Mayer path homology case.

\begin{definition}
     A n-path $v=\sum v^{i_0,\cdots,i_p}e_{i_0,\cdots,i_n}$ is called $(a,b)$-cluster if all the elementary paths $e_{i_0,\cdots,i_n}$ with nonzero coefficients have $i_0=a$ and $i_n=b$. A path is called cluster if it is a $(a,b)$-cluster for some $a,b\in V(G)$.
\end{definition}

It is known that any $n$-path in $\Omega_n^2$ is a sum of $\Omega_n^2$ clusters \cite{GrigoryanSurvey} . Next lemma shows that this property can be lifted to Mayer setting as well.
\begin{lemma}
    Any element $v\in \Omega_n^{N,1}$ is a sum of $(a,b)$-clusters $v_{a,b}\in  \Omega_n^{N,1}$.
\end{lemma}
\begin{proof}
    Let $v\in \Omega_n^{N,1}$ then $\partial(v)\in \mathcal{A}_{n-1}$. Observe that $v$ can be written as a sum of $(a,b)$-clusters so that $v=\sum_{a,b\in V(G)}v_{a,b}$ where $v_{a,b}\in\mathcal{A}_n$.
    Observe that,
\[
\partial(e_{a,i_1,\cdots,i_{n-1},b}) 
= e_{i_1,\cdots,i_{n-1},b}+\xi^{n} e_{a,i_1,\cdots,i_{n-1}}
+\sum_{j=1}^{n-1}
\xi^{j} e_{a,i_1,\cdots,i_j^*\cdots ,i_{n-1},b}.
\]
Observe that $e_{a,i_1,\cdots,i_j^*\cdots ,i_{n-1},b}$ might be non-allowed. The cancellation can occur in two different ways such as $v-v'$ or $v+\sum_{m=1}^{N-1}\xi^m v_m$ where at least one $v_m\neq v$ and $v',v_m$ has the same non-allowed face. Both of the cancellation happens within the cluster since $v'$ and $v_m$ are in $(a,b)$-cluster.
\end{proof}
Let $w\in \Omega^{N,1}_3$ be called minimal element if no sub-linear combinations of its components are in $\Omega_3^{N,1}$. 
By [Lemma 2.4 of \cite{GrigoryanSurvey}], every $\Omega_n^{2,1}$-cluster is a sum of minimal $\Omega_n^{2,1}$-clusters which leads to the basis of $\Omega_n^{2,1}$ as stated at Proposition 2.5 at \cite{GrigoryanSurvey}. These results can be stated for $\Omega_n^{N,1}$.
\begin{figure}
    \centering
    \includegraphics[width=0.3\linewidth]{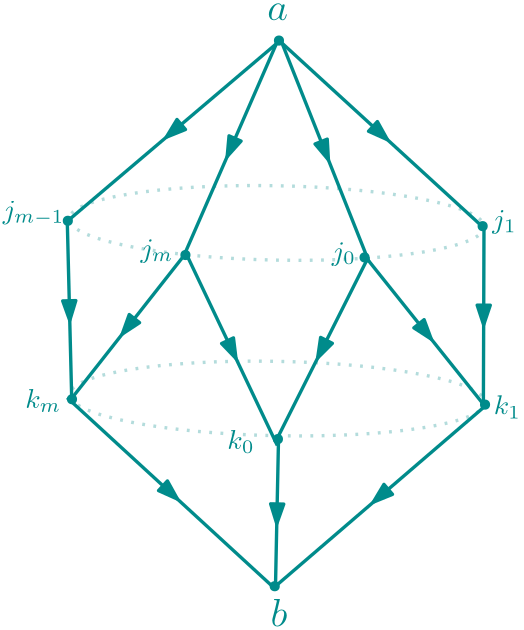}
    \caption{Trapezohedron $T_m$}
    \label{fig:Tm}
\end{figure}
The basis of $\Omega_3^{2}$ contains a core element which is called trapezohedron $T_m$ defined as for $m\geq 2$
\[ e_{a,j_t,k_t,b}+e_{a,j_{t+1},k_{t},b}\]
for all $t=0,\cdots, m-1\pmod m$
as in Fig \ref{fig:Tm}.

For completeness we will state the results from \cite{GrigoryanSurvey} and \cite{constructionomega3basis}. 
Let $G$ be a directed graph whose vertex set is partitioned into disjoint subsets $A_1, A_2,\cdots , A_n$.
Let $H$ be the directed graph with vertices $a_1, a_2,\cdots , a_n$. Define a map $f : G\to H$ by $f(x) = a_i$ for all $x \in A_i$. If
\[a_i\to a_j \text{ in } H \text{ if and only if there exist } x \in A_i \text{ and } y \in A_j \text{ such that } x \to y \text{ in } G\]
then the map $f$ is called a merging map.

For the next theorem, minimal $\partial$-invariant path means there is no nonempty proper subset of its components whose nonzero linear combination is also $\partial$-invariant. The following theorem was first stated at \cite{GrigoryanSurvey} as Theorem 2.10 under the condition that the digraph will be free of multisquares and double edge. Later at \cite{constructionomega3basis} these conditions are lifted.
\begin{theorem}
    Let $G$ digraph. Every minimal path in $\Omega_3^{2,1}$ is either trapezohedron or a merging image thereof. There is a basis of $\Omega_3^{2,1}$ containing trapezohedral paths $T_m$ with $m\geq 2$ with their merging images. 
    \end{theorem}
    The construction of the basis relies on the cluster-graph conversion. Let $V$ be the vertex set of all elementary paths $v=e_{a,i,j,b}$ from $(a,b)$-cluster. Construct a graph $\Gamma$ with $V=V(\Gamma)$ and the edge set is 
    \begin{itemize}
        \item the edge between $e_{a,i,j,b}$ and $e_{a,i',j,b}$ is colored $1$
        \item the edge between $e_{a,i,j,b}$ and $e_{a,i,j',b}$ is colored $2$
    \end{itemize}
    Let $E_i$ be the set of color $i$ edges where $i=1,2$.

    In the proof of basis for $\Omega_3^{2,1}$ at \cite{GrigoryanSurvey}, no multisquares imply that $\Gamma$ can be either a polygon or a line both of which describe a class of generators. Later in \cite{constructionomega3basis}, this condition is eliminated by considering the cycles and maximal alternating path on $\Gamma$. 

    In the our work, we will label the vertices in $\Gamma$ that results in a finer classification components which is needed to discover $\Omega_3^N$.

To organize the possible $3$-path contributions, we assign to each $v=e_{i,j,k,l}\in\mathcal{A}_3$ its image-type
\[
\phi(v)=(\phi(e_{j,k,l}),\phi(e_{i,k,l}),\phi(e_{i,j,l}),\phi(e_{i,j,k}))\in\{T,S,W,N_w\}^4= \Sigma^4,
\]
where $\Sigma = \{T,S,W,N_w\}$ denote triangle, square, admissible, and non-admissible types, respectively.

The indicator maps which record where the non-admissible faces occur in the image-type of a $3$-path defined as $\pi_r:\Sigma^4\to\{0,1\}$ by where $r=2,3$
\[
\pi_i(\alpha_1,\alpha_2,\alpha_3,\alpha_4)=
\begin{cases}
1,& \alpha_i=N_w,\\
0,& \text{otherwise},
\end{cases}
\]
For $v\in\mathcal{A}_3$ we also write $\pi_r(v):=\pi_r(\phi(v))$ for $r=2,3$.

\begin{lemma}\label{lemma:Nw-cancellation}
   Let $v\in\Omega_3^{N,1}$ be a cluster element where
$v=\sum_{i=1}^k \alpha_i v_i$ with $v_i=e_{a,j_i,k_i,b}\in\mathcal{A}_3$. 
If $\pi_2(v_i)=1$, then the fourth entry of $\phi(v)$ is of square type,
i.e.\ $\phi(e_{a,j_i,k_i})=S$.
Dually, if $\pi_3(v_i)=1$, then the first entry of $\phi(v)$ is of square type,
i.e.\ $\phi(e_{j_i,k_i,b})=S$.

\end{lemma}

\begin{proof}

Assume $v_i=e_{a,j_i,k_i,b}$ contains a non-admissible face which happens only at $e_{a,k_i,b}$ or $e_{a,j_i,b}$.
For $v\in\Omega^{N,1}_3$, every non-admissible term in $\partial( v)$ must be eliminated by pairing with another term of the same key. Such elimination can only occur through linear combination $3$-path that has the same non-admissible term in the same cluster. 

If $\pi_1(v_i)=1$, then $e_{a,k_i}\not\in \mathcal{A}_1$. Since non-admissible term is canceled, we know that there exists at least one $v_t\neq v_i$ so that $v_t= e_{a,j_t,k_i,b}$. Thus $e_{a,j_t,k_i}-e_{a,j_i,k_i}$ forms a square. The argument is analogous for $\pi_2(v_i)=1$.

\end{proof}

The above lemma shows that cancellation of non-admissible faces is highly constrained: a non-admissible face can only be eliminated by pairing with another $3$-path that shares the same face, and such a pairing necessarily produces a square-type configuration.

In the construction of $\Gamma$, this cancellation mechanism is encoded globally by edges in the graph $\Gamma$, and generators arise from alternating paths and cycles in $\Gamma$. In our formulation, this same cancellation mechanism is encoded locally through the image-type $\phi(v)$ and the indicator maps $\pi_2,\pi_3$.

Therefore, instead of searching for alternating paths in $\Gamma$, we may identify generators by determining which collections of $3$-path types can be connected through successive cancellations of non-admissible faces.

\begin{theorem}
\label{thm:C_N1_types}
    Let $\mathcal{C}_{N,1}$ be the set of all $3$-paths components that occur in minimal cluster elements of $\Omega_3^{N,1}$.
    \[\mathcal{C}_{N,1}=\{v\in\mathcal{A}_3|\ \exists \text{
    minimal cluster }w\in \Omega^{N,1}_3 \text{ so that } w=\alpha_v v+\sum\alpha_{v_i} v_i\ ,  v_i\in \mathcal{A}_3\}\] 
    
    Let $\phi:\mathcal{C}_{N,1}\to \{T,S,W,N_w\}^4$ be the image-type map. Then, for $N\ge 2$,
\[
\phi(\mathcal{C}_{N,1})=\{(T,T,T,T),(T,S,S,T)\}\ \cup\ \mathcal{L},
\]
where $\mathcal{L}$ consists of exactly seven patterns, namely
\[
(S,T,N_w,T), (S,S,N_w,T), (S,W,N_w,T), (T,N_w,T,S), \]
\[(T,N_w,S,S), (T,N_w,W,S), (S,N_w,N_w,S)
\]
In particular, $|\phi(\mathcal{C}_{N,1})|=9$.
\end{theorem}

\begin{proof}
Let $v\in \mathcal{C}_{N,1}$. Then $v=e_{i,j,k,l}\in\mathcal{A}_3$ occurs as a component of a minimal
$w\in\Omega^{N,1}_3$. By definition,
\[
\phi(v)=\big(\phi(e_{j,k,l}),\ \phi(e_{i,k,l}),\ \phi(e_{i,j,l}),\ \phi(e_{i,j,k})\big)\in\Sigma^4,
\qquad \Sigma=\{T,S,W,N_w\}.
\]

The type of each $2$-face is determined by the existence of the shortcut edges
$e_{i,k}$, $e_{j,l}$, and $e_{i,l}$. We distinguish cases according to membership of these edges in
$\mathcal{A}_1$.

\begin{itemize}
\item If $e_{i,k},e_{j,l},e_{i,l}\in\mathcal{A}_1$, then every $2$-face is triangular and hence
$\phi(v)=(T,T,T,T)$, in which case $v\in\Omega^{N,1}_3$ for all $N\ge2$.

\item If $e_{i,k},e_{j,l}\in\mathcal{A}_1$ but $e_{i,l}\notin\mathcal{A}_1$, then for $N=2$ the only
possible configuration on four vertices produces the internal square and yields
$\phi(v)=(T,S,S,T)$ which implies $v\in \Omega_3^{2,1}$. For $N\ge3$ this configuration cannot be realized in $ \Omega_3^{N,1}$ using only the four vertices
$\{i,j,k,l\}$ (see example \ref{diamond_example}) but with additional vertices that will create squares with $e_{i,*,k}$; the remaining possibility on four vertices would force the pattern $(T,W,W,T)$,
does not belong to $\Omega^{N,1}_3$. Thus in this case the only pattern in $\mathcal{C}_{N,1}$ is
$(T,S,S,T)$  for $N\geq 2$.

\item If $e_{i,k}\in\mathcal{A}_1$ and $e_{j,l}\notin\mathcal{A}_1$, then the third face
$e_{i,j,l}$ is non-admissible, so $\pi_3(v)=1 $, and one obtains
\[
\phi(v)\in \{S,W\}\times\{T,S,W\}\times\{N_w\}\times\{T\}.
\]

\item If $e_{j,l}\in\mathcal{A}_1$ and $e_{i,k}\notin\mathcal{A}_1$, then the second face
$e_{i,k,l}$ is non-admissible, so the second coordinate of $\phi(v)$ is $N_w$, and one obtains
\[
\phi(v)\in \{T\}\times\{N_w\}\times\{T,S,W\}\times\{S,W\}.
\]

\item If $e_{i,k},e_{j,l}\notin\mathcal{A}_1$, then the second and third faces are non-admissible and
\[
\phi(v)\in \{S,W\}\times\{N_w\}\times\{N_w\}\times\{S,W\}.
\]
\end{itemize}

So far this yields a finite list of candidate patterns, but most of them are excluded by the
cancellation requirement defining $\Omega^{N,1}_3$ of a simple digraph.

By Lemma \ref{lemma:Nw-cancellation}, all patterns whose second entry is $N_w$ but fourth entry is not $S$, and all patterns whose third entry is $N_w$ but
first entry is not $S$ are not in $\mathcal{C}_{N,1}$.
The surviving patterns are
\[
(S,T,N_w,T),\ (S,S,N_w,T),\ (S,W,N_w,T),\ (T,N_w,T,S),\]
\[ (T,N_w,S,S),\ (T,N_w,W,S),\ (S,N_w,N_w,S)
\]
Together with $(T,T,T,T)$ (for all $N\ge2$) and $(T,S,S,T)$, this gives $|\phi(\mathcal{C}_{N,1})|=9$ for $N\ge2$.
\end{proof}
The significance of this classification is that every component of a minimal element of $\Omega_3^{N,1}$ belongs to one of finitely many types. Consequently, the problem of identifying generators reduces to determining which sequences of these types can be combined so that all non-admissible faces cancel.

Each vertex of $\Gamma$ can be labeled by its type that is listed in the previous theorem. We determine the connectivity relations between types, which encode all possible cancellations. We label component types by
\[
\begin{aligned}
\gamma_1&=(S,T,N_w,T), & \gamma_2&=(S,S,N_w,T), & \gamma_3&=(S,W,N_w,T),\\
\gamma_4&=(T,N_w,T,S), & \gamma_5&=(T,N_w,S,S), & \gamma_6&=(T,N_w,W,S),\\
\gamma_7&=(S,N_w,N_w,S).
\end{aligned}
\]
The non-admissible-free elements labeled as follows
\[\gamma_8=(T,S,S,T),  \gamma_9=(T,T,T,T)\]
For the following lemma, one-step connectable means they share the $N_w$ face and the linear combination of these two elements will cancel the $N_w$ component.

\begin{lemma}\label{lem:connectability_structure}
The components $\gamma_i$ are one-step connectable only to themselves and to $\gamma_7$ for $i=1,2,4,5$ where $\gamma_3$ and $\gamma_6$ is one-step connectable to only $\gamma_7$. The component $\gamma_7$ is one-step connectable to all elements of $\mathcal{L}$.
\end{lemma}
\begin{proof}
    Let $w\in \Omega_3^{N,1}$ be a $(a,b)$-cluster so that $w=\sum w_i$.
    The existence of $e_{a,b}$ will create partition of patterns. Thus the elements of $N_{a,b}=\{\gamma_2,\gamma_3,\gamma_5,\gamma_6\}$ is not one-step connectable to the elements of $A_{a,b}=\{\gamma_1,\gamma_4\}$ where $A_{a,b}$ is the set of types where $e_{a,b}$ is admissible and $N_{a,b}$ is non-admissible. 
    Furthermore, they will not be in the linear combination together with $\gamma_1$ and $\gamma_4$ for generators of $\Omega_3^{N,1}$.

    Similarly, the existence of multiple $e_{a,*,b}$ is possible for the patterns $\gamma_2,\gamma_5$ while there is a unique $e_{a,*,b}$ in $\gamma_3,\gamma_6$. This results in three different groups such as $\{\gamma_1,\gamma_4\},\{\gamma_2,\gamma_5\}$ and $\{\gamma_3,\gamma_6\}$

    Observe that $\pi_2(\gamma_1)=0$ and $\pi_2(\gamma_4)=1$ which makes them incompatible for one-step connection where $\pi_2$ and $\pi_3$ are indicator maps that is defined in the proof of Theorem \ref{thm:C_N1_types}. The only element of $\mathcal L$ whose image-type contains non-admissible faces at both indices $2$ and $3$ is $\gamma_7$. Therefore $\gamma_1$ and $\gamma_4$ are one-step connectable only to themselves and to $\gamma_7$. The same reason is valid for pairs  $(\gamma_3,\gamma_6)$ and $(\gamma_2,\gamma_5)$

    Let $\phi(e_{a,j_1,k_1,b})=\gamma_3$ and $\phi(e_{a,j_1,k_2,b})=\gamma_3$ be two $3$-paths so that their $N_w$ components align where $\gamma_3=(S,W,N_w,T)$. The pattern induces that $e_{a,k_2}\in\mathcal{A}_1$ and $e_{a,k_1,b}$ is only admissible which conditions that there is no $e_{i,l}$ or $e_{i,*,l}$. Since $e_{i,k_2}\in\mathcal{A}_1$ we have $e_{i,k_2,l}\in\mathcal{A}_2$. Therefore this one-step connection is not possible. Thus $\gamma_3$ is only one-step connectable to $\gamma_7$. 
    Since $\gamma_6=(T,N_w,W,S)$ is symmetric of $\gamma_3=(S,W,N_w,T)$, the same reason is an obstacle to create one-step connection $\gamma_6$ to itself. Therefore $\gamma_6$ is only one-step connectable to $\gamma_7$. There is no restriction to connect $\gamma_7$ to itself which proves the last part of the claim.
    
 \end{proof}

The preceding lemma determines exactly which pairs of types can cancel a given non-admissible face. Thus, the one-step connectability relation plays the same role as the edge relations $E_1$ and $E_2$ in the graph $\Gamma$. In particular, sequences of one-step connectable types correspond to alternating paths in $\Gamma$. Therefore, the coexistence relation together with the connectivity relations on the $\gamma_i$ generates the equivalence classes of building blocks for minimal clusters.

 \begin{theorem}
     The minimal clusters in $\Omega_3^2$ are formed by the following groups
     \[T_1=(\gamma_1,\gamma_4,\gamma_7), T_2=(\gamma_2,\gamma_5,\gamma_7), T_3=(\gamma_3,\gamma_6,\gamma_7),T_4=(\gamma_7),  T_5=(\gamma_9), T_6 = (\gamma_8). \]
 \end{theorem}
\begin{proof}
Let $w\in \Omega_3^{2}$ be a minimal cluster. By Theorem
\ref{thm:C_N1_types}, every component of $w$ has image-type equal to one of
\[
\gamma_1,\ldots,\gamma_9.
\]
The types $\gamma_8=(T,S,S,T)$ and $\gamma_9=(T,T,T,T)$ contain no
non-admissible face. Hence such components are already $\partial$-invariant
inside their own cluster. Therefore, by minimality, if one of these types
occurs in $w$, then no other type can occur in the same minimal cluster.
This gives the two isolated classes
\[
T_5=(\gamma_9), \qquad T_6=(\gamma_8).
\]

It remains to consider the types in
\[
\mathcal L=\{\gamma_1,\ldots,\gamma_7\}.
\]
These are precisely the types containing at least one non-admissible face.
Since $w\in \Omega_3^2$, every non-admissible face appearing in
$\partial w$ must cancel with the same non-admissible face coming from another
component of $w$. Thus the components of $w$ must be connected by successive
one-step connectability relations.

By Lemma \ref{lem:connectability_structure}, the only possible one-step
connections are the following:
\[
\gamma_1,\gamma_2,\gamma_4,\gamma_5
\quad\text{are connectable only to themselves and to }\gamma_7,
\]
\[
\gamma_3,\gamma_6
\quad\text{are connectable only to }\gamma_7,
\]
and
\[
\gamma_7
\quad\text{is connectable to every element of }\mathcal L.
\]
Therefore, any connected minimal cluster containing non-admissible faces must
belong to one of the connected components generated by these relations.

Now the coexistence conditions separate the possible types into three
families. The types $\gamma_1$ and $\gamma_4$ coexist with $\gamma_7$ since they admit $e_{a,b}$ as an edge, giving
\[
T_1=(\gamma_1,\gamma_4,\gamma_7).
\]
Similarly, the types $\gamma_2$ and $\gamma_5$ coexist with $\gamma_7$ and
with each other since $e_{a,b}$ is not an edge and there can be multiple $e_{a,*,b}$ 2 paths, giving
\[
T_2=(\gamma_2,\gamma_5,\gamma_7).
\]
The types $\gamma_3$ and $\gamma_6$ cannot connect to themselves, but each
is one-step connectable to $\gamma_7$. Hence they form the third possible
family where $e_{a,b}$ is not an edge and there exists an unique 2-paths $e_{a,*,b}$
\[
T_3=(\gamma_3,\gamma_6,\gamma_7).
\]
Finally, since $\gamma_7$ contains non-admissible faces in both relevant
positions and is one-step connectable to itself, it may form a minimal
cluster without any other type. This gives
\[
T_4=(\gamma_7).
\]

No other minimal cluster can occur. Indeed, any minimal cluster must be
connected under the one-step connectability relation; otherwise it would
split into two nonempty subclusters whose boundaries are separately
allowable, contradicting minimality.
Together with the two non-admissible-free isolated types $\gamma_8$ and
$\gamma_9$, this proves the claim.
\end{proof}

The following corollary will describe the exact combinations of $\gamma_i$ appearing in the minimal cluster.

\begin{corollary}
The minimal clusters in $\Omega_3^{N,1}$ are exactly the following, for $m\geq 0$:

\begin{itemize}
    \item For $i=1,2$, $\gamma_7$-chains of the form
    \[
    \gamma_i \;-\; (\gamma_7)^m \;-\; \gamma_j,
    \]
    where $j=i$ if $m$ is even and $j=i+3$ if $m$ is odd.
    \item For $i=3$, $\gamma_7$-chains of the form
    \[
    \gamma_i \;-\; (\gamma_7)^m \;-\; \gamma_j,
    \]
    where $j=i$ if $m\geq 2$ is even  and $j=i+3$ if $m$ is odd.
    \item Pure $\gamma_7$-clusters trapezohedrons and the isolated types $\gamma_8$ and $\gamma_9$.

\end{itemize}
\end{corollary}
\begin{proof}
By the classification theorem, every minimal cluster lies in one of the
families
\[
(\gamma_i,\gamma_{i+3},\gamma_7), \quad i=1,2,3,
\]
or is one of the isolated types $\gamma_7,\gamma_8,\gamma_9$.

By Lemma \ref{lem:connectability_structure}, the only way to connect
$\gamma_i$ and $\gamma_{i+3}$ is through $\gamma_7$. Hence every nontrivial
minimal cluster in these families is a chain whose endpoints lie in
$\{\gamma_i,\gamma_{i+3}\}$ and whose intermediate vertices are all $\gamma_7$.

Each $\gamma_7$-connection switches the endpoint type. Therefore, if the
number $m$ of intermediate $\gamma_7$-vertices is even, the endpoints
coincide, while if $m$ is odd, the endpoints differ. This yields the stated
chains.

Minimality forces the chain to terminate exactly at the first cancellation,
so no other configurations occur.

Finally, $\gamma_8$ and $\gamma_9$ contain no non-admissible faces and are
therefore isolated. A cluster consisting only of $\gamma_7$ must cancel both
types of non-admissible faces which forms trapezohedron $T_m$.
\end{proof}
\begin{figure}
    \centering
    \includegraphics[width=0.9\linewidth]{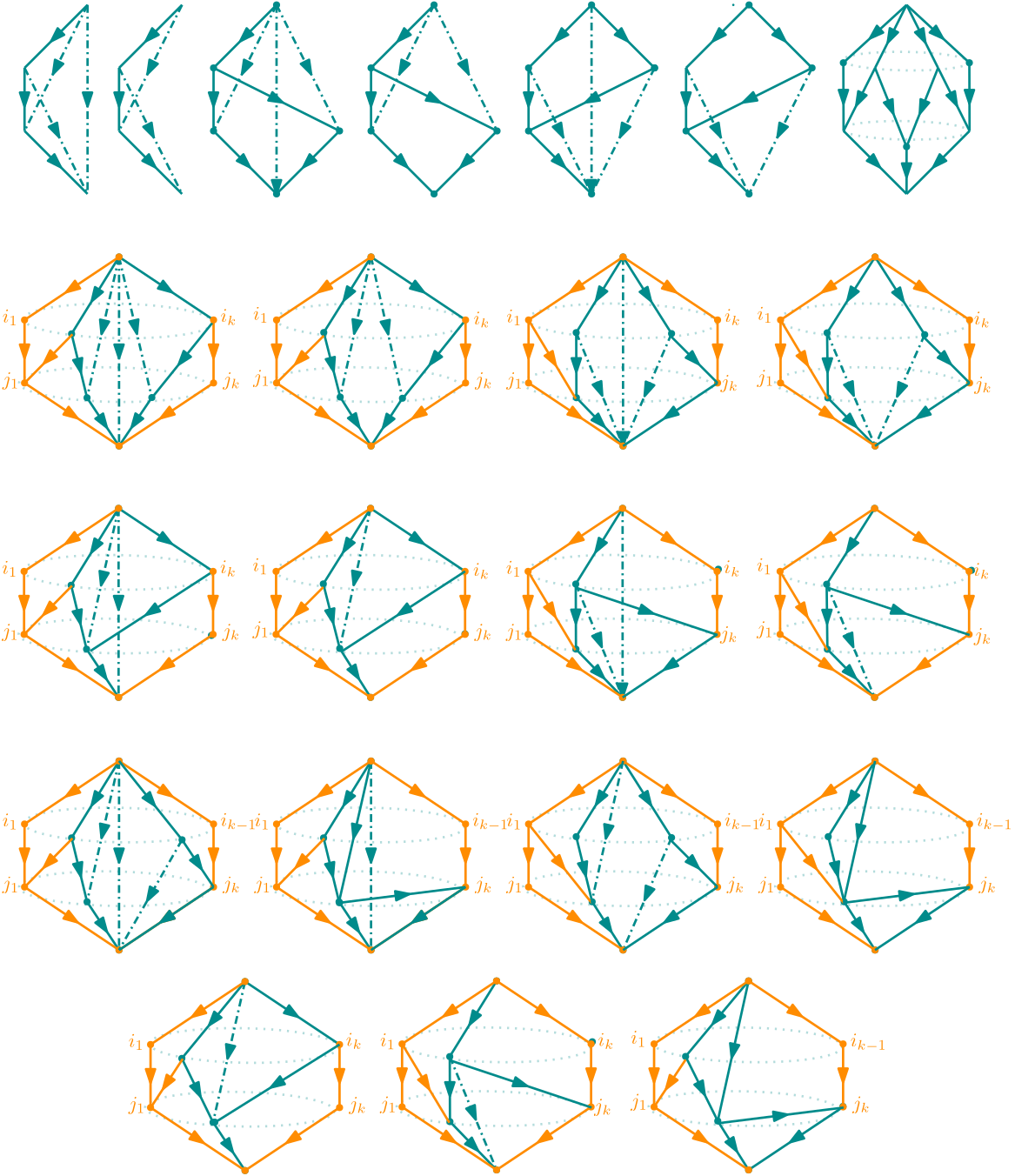}
    \caption{Left to right.
Top row: $\gamma_9$, $\gamma_8$, self-connections $(\gamma_i,\gamma_i)$ for $i=1,2,4,5$, and the trapezohedron.
Second row: connections $(\gamma_i,\gamma_7)$ introducing a new vertex for $i=1,2,4,5$.
Third row: connections $(\gamma_i,\gamma_7)$ with repeated vertices for $i=1,2,4,5$.
Fourth row: $(\gamma_i,\gamma_{i+3},\gamma_7)$ with a new vertex and with a repeated vertex for $i=1,2$.
Last row: connections $(\gamma_3,\gamma_7)$, $(\gamma_6,\gamma_7)$, and the chain $(\gamma_3,\gamma_6,\gamma_7)$.
The dashed edges represent admissible edges that are not involved in the generator element representation.
}
    \label{fig:all_class}
\end{figure}

For the families $T_1$ and $T_2$, there are two distinct realizations. These correspond to the two possible endpoint configurations of the alternating path. In one case, the endpoint components share three vertices, while in the other case they share only two vertices, even though no edge exists between them in $\Gamma$. Each configuration produces a distinct generator $(a,b)$-cluster.

For the family $T_3$, there is only one possible realization. This is due to the fact that the types $\gamma_3$ and $\gamma_6$ contain a $W$-component, which forces the existence of a unique $2$-path of the form $e_{a,*,b}$, yielding a unique generator pattern. 
Together, these configurations justify that every minimal cluster is generated by $\gamma_7$-chains connecting compatible endpoint types, as described in the preceding classification as in the Fig \ref{fig:all_class}.

Our goal is to determine $\Omega_3^{N}$, which is a subspace of $\Omega_3^{N,1}$. The next example shows that not all generators of $\Omega_3^{N,1}$ belong to $\Omega_3^{N,2}$, and an additional selection criterion is required.

\begin{figure}
    \centering
    \includegraphics[width=0.3\linewidth]{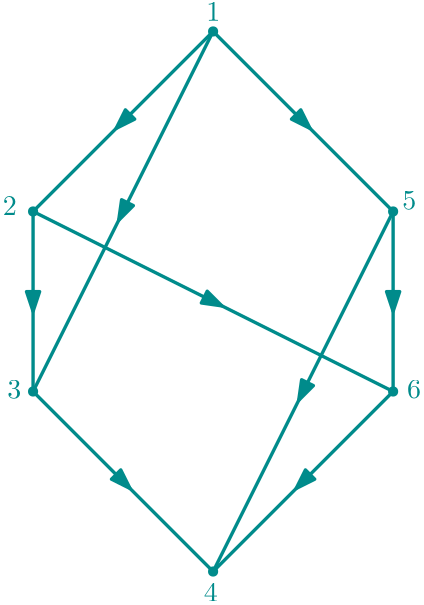}
    \caption{The digraph in the Example \ref{ex:counterexamplev2v5}}
    \label{fig:v2v5example}
\end{figure}

\begin{example}\label{ex:counterexamplev2v5}
    Let $G$ be the graph of the following combinations $\gamma_2-\gamma_7-\gamma_5$ as in Fig \ref{fig:v2v5example} 

    For $N_w$ terms to be canceled the minimal cluster element will have the linear combination as follows $w= e_{1234}-e_{126,4}+e_{15,6,4}$.
    \[\partial(w)= e_{234}-e_{264}+e_{564}+\xi e_{13,4}+\xi^2 e_{15,4} +\xi^3( e_{123}-e_{126}+e_{15,6})\in \mathcal{A}_2\]
    \[\partial^2(w)= (\xi^2+\xi^3)e_{14}+v,\quad v\in\mathcal{A}_1\]
    For $N\geq 3$, $\xi^2+\xi^3\neq 0$ therefore $w\in\Omega_3^{N,1}$ but $w\not\in \Omega_3^{N,2}$.
\end{example}

 The following definition will create an identification of the non-admissible edges that arises in the boundary of different image types such as square and admissible.
\begin{definition}
    Let $(i,k)\in V \times V $ so that $e_{i,k}\not\in\mathcal{A}_1$.
    \begin{itemize}
        \item $e_{i,k}$ is called connecting edge if there exists a square $e_{i,j,k}-e_{i,j',k}\in \Omega_2^{N}$.
        \item $e_{i,k}$ is called complementary edge if there is only one $e_{i,j,k}\in \mathcal{A}_2$.
    \end{itemize}
    
\end{definition}

 If $v\in\mathcal{A}_3$, then \[\partial^2(v)= \sum_k s_k e^s_{i,j}+\sum_l u_l e_{i,j}^c +w \]
where $w\in\mathcal{A}_1$, $e_{i,j}^s$ is connecting edge and  $e_{i,j}^c$ is complementary edge. 
When $N=2$, these vanishes since $\partial^2=0$ but for $N\geq 3$, these components need to be eliminated. 
In other words the generators should be free of the connecting edges and complementary edge.

\begin{prop}

\label{v2v5v8comb}
The types supported on $(\gamma_2,\gamma_5,\gamma_7)$ and $\gamma_8$ give minimal elements that do not lie in $\Omega_3^{N,2}$ for $N\geq 3$.
\end{prop}

\begin{proof}
For $N\geq 3$, cancellation of non-admissible faces in a minimal cluster must occur in such a way that all induced connecting edges are also eliminated.

First observe that in the families generated by $(\gamma_1,\gamma_4,\gamma_7)$ and $(\gamma_3,\gamma_6,\gamma_7)$, the cancellation of $N_w$-faces produces pairs of square-type faces with matching coefficients. As a consequence, the corresponding connecting edges cancel in the boundary, and no residual terms remain. In particular, in the case of $(\gamma_3,\gamma_6,\gamma_7)$, the terminal components share the same $W$-type face, and their coefficients ensure cancellation of the complementary edges.

In contrast, for the family $(\gamma_2,\gamma_5,\gamma_7)$ and for the isolated type $\gamma_8$, the presence of square-type faces in the interior prevents such cancellation. More precisely, the square-type contributions do not match in a way that eliminates the induced connecting edges. As a result, nonzero boundary terms remain, and these elements fail to lie in $\Omega_3^{N,2}$.

Therefore, minimal elements supported on $(\gamma_2,\gamma_5,\gamma_7)$ and $\gamma_8$ are not contained in $\Omega_3^{N,2}$ for $N\geq 3$.
\end{proof}
    
\begin{theorem}
\label{omega_3n2}
If a $(a,b)$-cluster contains both $(\gamma_2,\gamma_5,\gamma_7)$ and $\gamma_8$ components, then their combination yields a generator in $\Omega_3^{N,2}\cap \Omega_3^{N,1}$. Consequently, the generators of $\Omega_3^{N,2}\cap \Omega_3^{N,1}$ are
\[
T_1=(\gamma_1,\gamma_4,\gamma_7),\quad
T_2=(\gamma_2,\gamma_5,\gamma_7,\gamma_8),\quad
T_3=(\gamma_3,\gamma_6,\gamma_7),\quad
T_4=(\gamma_7),\quad
T_5=(\gamma_9).
\]
\end{theorem}

\begin{proof}
Assume there exist two $(a,b)$-clusters, one supported on $(\gamma_2,\gamma_5,\gamma_7)$ and the other on $\gamma_8$. By construction, both clusters produce the same boundary term $e_{a,b}$ under $\partial^2$. Hence, by choosing appropriate coefficients, these contributions cancel in the linear combination, and the resulting element lies in $\Omega_3^{N,2}$. Since neither component alone belongs to $\Omega_3^{N,2}$ (by Proposition~\ref{v2v5v8comb}), their combination yields a minimal generator.

By Proposition~\ref{v2v5v8comb}, the types supported on $(\gamma_2,\gamma_5,\gamma_7)$ and $\gamma_8$ do not individually lie in $\Omega_3^{N,2}$, while the families
\[
T_1=(\gamma_1,\gamma_4,\gamma_7),\quad
T_3=(\gamma_3,\gamma_6,\gamma_7),\quad
T_4=(\gamma_7),\quad
T_5=(\gamma_9)
\]
already produce elements in $\Omega_3^{N,2}$.

Therefore, the only additional generators arise from combining $(\gamma_2,\gamma_5,\gamma_7)$ with $\gamma_8$, which yields the family
\[
T_2=(\gamma_2,\gamma_5,\gamma_7,\gamma_8).
\]
No other combinations produce new generators, since all other types either already lie in $\Omega_3^{N,2}$ or fail to cancel the boundary terms.

\end{proof}
\begin{corollary}
The space $\Omega_3^N$ is generated by those generators of $\Omega_3^{N,2}\cap \Omega_3^{N,1}$ in Theorem \ref{omega_3n2}. Equivalently,
\[
\Omega_3^N=\Omega_3^{N,1}\cap \Omega_3^{N,2}.
\]
\end{corollary}

The previous corollary is an observation of the fact $\Omega_3^{N,q}=\mathcal{A}_3$ for $q\geq 3$.

\subsection{The First Mayer Path Homology Group $H_1^{N,q}$}
Computation of standard path homology in degree $1$ is discussed in \cite{efficientH1} by Dey et al. It is shown that when $N=2$, the kernel of the boundary map $\partial_1:\Omega^2_1\to \Omega^2_0$ is generated by cycles in the underlying undirected graph of a given digraph.

More precisely, if $G_u$ denotes the underlying undirected graph of $G$, then a basis of the $1$-cycle space can be obtained from a spanning tree of $G_u$. Each non-tree edge determines a unique undirected cycle, and these cycles generate $\ker(\partial_1)$ via symmetric difference in mod 2 when coefficients are taken in from any field $\mathbb{K}$.

The following example demonstrates that this phenomenon does not extend to higher roots of unity. In particular, when $N\geq 3$, not every undirected cycle lies in the kernel of $\partial_1:\Omega_1^N\to \Omega_0^N$.

In Example \ref{nonisomorhiccomparison}, $L_1,L_2$
 and $L_3$ are three digraphs. For $N=2$, all undirected cycles forms kernel of $\partial_1$. However, when $N=3$, $Z_1^{3,1}(L_1)=0$ while others are nonzero. 
 
 Since $Z_1^{N,q}=\Omega_1^{N}$ for $q\geq 2$, we focus our investigation of $Z_1^{N,1}= Ker(\partial: \Omega_1^{N}\to \Omega_0^{N}) $. Let $G$ be a digraph and let $G_u$ denote its underlying undirected graph. 

Given a cycle
\[
i_1,i_2,\dots,i_n,i_1
\]
in $G_u$ with distinct vertices, we associate to it a $1$-path in $\Omega_1(G)$ by assigning to each edge the orientation induced from $G$, namely
\[
v_j=
\begin{cases}
e_{i_j,i_{j+1}}, & \text{if } i_j\to i_{j+1}\text{ is an edge of }G,\\[4pt]
e_{i_{j+1},i_j}, & \text{otherwise.}
\end{cases}
\]
A linear combination
\[
\sum_{j=1}^n y_j v_j,\qquad y_j\in\mathbb{C},
\]
is called a weighted cycle associated to the undirected cycle.

\medskip

The following proposition characterizes when such a weighted cycle lies in the kernel.

\begin{prop}\label{admissiblecycle}
Let $\xi$ be a primitive $N$th root of unity, $N\ge 3$, and let
\[
i_1,i_2,\dots,i_n,i_1
\]
be a cycle in $G_u$. For each $j$, let $x_j=\xi$ be if the orientation of the edge $v_j$ is the same with $(i_j,i_{j+1})$ otherwise $x_j=1$, and let $x_j^*$ satisfy $x_j x_j^*=\xi$.

Define coefficients $y_1,\dots,y_n$ by choosing $y_1\neq 0$ and setting
\[
y_j = -\frac{x_{j-1}^*}{x_j} \, y_{j-1}, \qquad j=2,\dots,n.
\]

Then
\[
\sum_{j=1}^n y_j v_j \in \ker(\partial_1:\Omega_1^N\to\Omega_0^N)
\]
if and only if one of the following holds:
\begin{enumerate}
    \item $n$ is even and $u_1-u_2\equiv 0 \pmod N$,
    \item $n$ is odd, $N$ is even, and $u_1-u_2\equiv \frac{N}{2} \pmod N$,
\end{enumerate}
where $u_1$ and $u_2$ denote the number of indices $j$ for which
\[
\frac{x_j^*}{x_j}=\xi \quad \text{and} \quad \frac{x_j^*}{x_j}=\xi^{-1},
\]
respectively.
\end{prop}

\begin{proof}
Let $A$ be the matrix associated to the above cycle, defined by
\[
A_{i,i}=x_i,\qquad A_{i+1,i}=x_i^*,\qquad A_{1,n}=x_n^*,
\]
with indices modulo $n$, where $x_i\in\{1,\xi\},\ x_i x_i^*=\xi$, and all other entries of $A$ are zero.

Suppose there exists a nonzero vector $y=(y_1,\dots,y_n)^T\in\mathbb C^n$ such that $Ay=0$. Then
\[
x_1y_1+x_n^*y_n=0,\qquad x_{i-1}^*y_{i-1}+x_i y_i=0,\qquad i=2,\dots,n.
\]

From the latter equations we obtain the recursion
\[
y_i=-\frac{x_{i-1}^*}{x_i}y_{i-1},\ i=2,\dots,n
\Rightarrow y_n=(-1)^{n-1}\prod_{i=2}^n \frac{x_{i-1}^*}{x_i}\,y_1\]

Substituting into the first equation and using $y_1\neq 0$, we obtain
\[
0=x_1y_1+x_n^*y_n
=x_1y_1+(-1)^{n-1}x_n^*\prod_{i=2}^n\frac{x_{i-1}^*}{x_i}\,y_1,
\]
which simplifies to
\[
(-1)^n\prod_{i=1}^n\frac{x_i^*}{x_i}=1.
\]

Observe that for each $i$, $\frac{x_i^*}{x_i}\in\{\xi,\xi^{-1}\}$. Let $u_1$ denote the number of indices $i$ for which $\frac{x_i^*}{x_i}=\xi$, and let $u_2$ denote the number of indices $i$ for which $\frac{x_i^*}{x_i}=\xi^{-1}$.

Then
\[
u_1+u_2=n,\qquad \prod_{i=1}^n\frac{x_i^*}{x_i}=\xi^{u_1-u_2}.
\]
Hence the kernel condition becomes
\[
(-1)^n\xi^{u_1-u_2}=1
\]

If $n$ is even, this reduces to
\[
\xi^{u_1-u_2}=1 \Rightarrow u_1-u_2\equiv 0 \pmod N
\]

If $n$ is odd, we obtain
\[
\xi^{u_1-u_2}=-1.
\]
This is possible only when $N$ is even, in which case
\[
u_1-u_2\equiv \frac N2 \pmod N.
\]

Conversely, if these conditions hold, the recursion defines a nonzero vector $y$ satisfying $Ay=0$, and hence
\[
\sum_{j=1}^n y_j v_j \in \ker(\partial_1).
\]

For the final assertions, note that if $n$ is odd, then
\[
u_1-u_2 = n - 2u_2
\]
is also odd. Hence $\frac N2$ must be odd, which is equivalent to $N\equiv 2\pmod 4$. Moreover, since $|u_1-u_2|\le n$, the congruence
\[
u_1-u_2\equiv \frac N2\pmod N
\]
can hold only if $n\ge \frac N2$.

\end{proof}
The linear combinations arising from undirected cycles in $G_u$ that satisfy the conditions of the previous proposition are called admissible cycles. However, admissible cycles do not exhaust all elements of the kernel. In particular, certain linear combinations of non-admissible cycles may also lie in $\ker(\partial_1)$. 
\begin{definition}
Let $I$ and $J$ be two undirected cycles in the same connected component of a directed graph $G$. Let $C$ be any undirected path connecting a vertex of $I$ to a vertex of $J$. The merge of $I$ and $J$ with respect to $C$ is the subgraph
\[ I \vee_C J
=
\left(
V(I)\cup V(J)\cup V(C),\;
E(I)\cup E(J)\cup E(C)
\right)\]
If $I$ and $J$ have non-empty intersection then
\[ I \vee J
=
\left(
V(I)\cup V(J),\;
E(I)\cup E(J)
\right)\]
\end{definition}

\begin{prop}\label{nonadmissiblecomb}

Let $I$ and $J$ be two non-admissible cycles in $G_u$ within the same connected component with $C$ be the shortest connecting undirected path. Then there exists nonzero weighted element supported on the merge $ I \vee_C J$ that lies in $\ker(\partial_1)$.

\end{prop}
\begin{proof}
    Let $I$ and $J$ be two non-admissible cycles within the same connected component induced by the sequence $i_1j_1\cdots j_m$ and $i_kl_1\cdots l_n$. Let $C$ be shortest connecting path from $i_1\in V(I)$ to $i_k \in V(J)$.
    Define the paths
\[
P_1 = (i_1,\dots,i_k), \quad
P_2 = (i_1,j_1\dots,j_m,i_1), \quad
P_3 = (i_k,l_1,\dots,l_n,i_k).
\]
       
    Let $a_1=k-1$, $a_2=m+1$ and $a_3=n+1$. Let $A_i$ be the adjacency matrix for $P_i$ where $i=1,2,3$.
    \[[A_i]_{j,j}=(x_i)_j\quad [A_i]_{j+1,j}=(x_i)_j^* \quad j=1,\cdots ,a_i \quad [A_2]_{1,a_2}=(x_2)^*_{a_2}\quad [A_3]_{1,a_3}=(x_3)^*_{a_3}\]
    where $(x_i)_j\in \{1,\xi\}$ and $(x_i)_j(x_i)_j^*=1$ for $\xi$ is the $N$th root of unity.
    Let $y_i=(y_{i,1},\cdots, y_{i,a_i})^T\in\mathbb{C}^{a_i}$ for $i=1,2,3$.
    Let $A$ be the adjacency matrix of merge of $I$ and $J$. $Ay^T=0$ yield the following equations where $y$ is stack of $y_i$
    \[y_{i,j}(x_i)_{j}^*+y_{i,j+1}(x_i)_{j+1}=0,\quad  j=1,\cdots, a_i-1,\quad i=1,2,3\quad(1)\]
    \[y_{i,1}(x_{i})_1+y_{i,a_i}(x_{i})_{a_i}^*=0\quad  i=2,3\quad (2)\]
    \[y_{1,1}(x_1)_1+y_{2,a_2-1}(x_{2})_{a_2-1}^*+y_{2,a_2}(x_2)_{a_2}=0 \quad (3)\]
    \[y_{1,a_1}(x_1)^*_{a_1}+y_{3,a_3-1}(x_{3})_{a_3-1}^*+y_{3,a_3}(x_3)_{a_3}=0 \quad (4)\]

    Recursive relation (1) implies for $i=1,2,3$
    \[y_{i,a_i-1}=(-1)^{a_i-2}y_{i,1}\dfrac{(x_i)_1^*}{(x_{i})_{a_i-1}}\prod_{j=2}^{a_i-2}\dfrac{(x_i)_j^*}{(x_i)_j}\]
    Therefore $(3)$ and $(4)$ can be rewritten as
    \[y_{1,1}(x_1)_1+y_{2,1}\left( (-1)^{a_2}(x_2)_1^* \prod_{j=2}^{a_2-1}\dfrac{(x_2)_j^*}{(x_2)_j}-\dfrac{(x_2)_1}{(x_2)_{a_2}^*}(x_2)_{a_2}\right)=0 \]
    \[(-1)^{a_1+1}y_{1,1}(x_1)_1^*\prod_{j=2}^{a_1}\dfrac{(x_1)_j^*}{(x_1)_j}+y_{3,1}(-1)^{a_3}(x_{a_3-1})^*\prod_{j=2}^{a_3-2}\dfrac{(x_3)_j^*}{(x_3)_j}\left(1-\dfrac{(x_3)_{a_3}}{(x_3)_{a_3-1}}\right)=0 \]
    These equations can be expressed as a linear system  $A'y'=0$ where $y'=(y_{1,1},y_{2,1},y_{3,1})^T$
    where the transpose of matrix $A'$ 
\[\begin{bmatrix}
        (x_1)_1 & (-1)^{a_1+1}(x_1)_1^*\prod_{j=2}^{a_1}\dfrac{(x_1)_j^*}{(x_1)_j}\\
        (-1)^{a_2}(x_2)_1^* \prod_{j=2}^{a_2-1}\dfrac{(x_2)_j^*}{(x_2)_j}-\dfrac{(x_2)_1(x_2)_{a_2}}{(x_2)_{a_2}^*} & 0 \\
        0 & (-1)^{a_3}(x_{a_3-1})^*\prod_{j=2}^{a_3-2}\dfrac{(x_3)_j^*}{(x_3)_j}\left(1-\dfrac{(x_3)_{a_3}}{(x_3)_{a_3-1}}\right)
    \end{bmatrix}\]

    Observe that 
    \[(-1)^{a_2}(x_2)_1^* \prod_{j=2}^{a_2-1}\dfrac{(x_2)_j^*}{(x_2)_j}-\dfrac{(x_2)_1(x_2)_{a_2-1}^*}{(x_2)_{a_2}^*}=0 \Rightarrow (-1)^{a_2}\prod_{j=1}^{a_2}\dfrac{(x_2)_j^*}{(x_2)_j}=1\]
     which means $P_2$ is an admissible. Thus the $2\times 2$ minor formed by the first two columns has have non-zero determinant. Hence $\text{rank}(A')=2$. Since $A'$ is a $2\times 3$ matrix, we obtain
\[
\dim \ker(A') = 3 - 2 = 1,
\]
so there exists a nontrivial solution $y' \neq 0$.\\
If $I$ and $J$ has nonzero intersection, then they are induced by vertex sequences $i_1,\dots,i_n$ and $j_1,\dots,j_m$, respectively, and suppose they share a nontrivial path $i_1 = j_1, \dots, i_k = j_k$.
 Define the paths and their adjacency matrix as $A_i$
\[
P_1 = (i_1,\dots,i_k), \quad
P_2 = (i_k,\dots,i_n,i_1), \quad
P_3 = (j_k,\dots,j_m,j_1).
\]

    Let $y_i=(y_{i,1},\cdots, y_{i,a_i})^T\in\mathbb{C}^{a_i}$ where $a_1=k-1$, $a_2=n-k+1$ and $a_3= m-k+1$.
    Let $A$ be the adjacency matrix of merge of $I$ and $J$. $Ay^T=0$ yield the following equations where $y$ is stack of $y_i$
    \[y_{i,j-1}(x_i)_{j-1}^*+y_{i,j}(x_i)_j=0\quad i=1,2,3\quad j=2,\cdots, a_i\quad (1)\]
    \[y_{1,1}(x_1)_1+y_{2,a_2}(x_2)_{a_2}^*+y_{3,a_3}(x_3)_{a_3}^*=0\quad (2)\]
    \[y_{1,a_1}(x_1)_{a_1}^*+y_{2,1}(x_2)_{1}+y_{3,1}(x_3)_1=0 \quad (3)\]
    The same recursive relation induced by (1). Thus $(2)$ and $(3)$ can be rewritten as
    \[y_{1,1}(x_1)_1+y_{2,1}(-1)^{a_2-1}\prod_{l=1}^{a_2-1}\dfrac{(x_2)_l^*}{(x_2)_{l+1}} (x_2)_{a_2}^*+y_{3,1}(-1)^{a_3-1}\prod_{l=1}^{a_3-1}\dfrac{(x_3)_l^*}{(x_3)_{l+1}} (x_3)_{a_3}^*=0\]
    \[y_{1,1}(-1)^{a_1-1}\prod_{l=1}^{a_1-1}\dfrac{(x_1)_l^*}{(x_1)_{l+1}} (x_1)_{a_1}^*+y_{2,1}(x_2)_{1}+y_{3,1}(x_3)_1=0 \]
   Substituting into (2) and (3) gives a system $A' y' = 0,
y' = (y_{1,1}, y_{2,1}, y_{3,1})^T,$ where traspose of $A'$ is
\[ \begin{bmatrix}
        (x_1)_1 & (-1)^{a_1-1}\prod_{l=1}^{a_1-1}\dfrac{(x_1)_l^*}{(x_1)_{l+1}} (x_1)_{a_1}^*\\
        (-1)^{a_2-1}\prod_{l=1}^{a_2-1}\dfrac{(x_2)_l^*}{(x_2)_{l+1}} (x_2)_{a_2}^* &(x_2)_{1}\\
        (-1)^{a_3-1}\prod_{l=1}^{a_3-1}\dfrac{(x_3)_l^*}{(x_3)_{l+1}} (x_3)_m^* & (x_3)_1
    \end{bmatrix}\]

     Consider the $2\times2$ minor formed by the first two columns. Its determinant is
\[
(x_1)_1 (x_2)_1
-
(-1)^{a_1+a_2}
\prod_{l=1}^{a_1}\frac{(x_1)_l^*}{(x_1)_l}
\prod_{l=1}^{a_2}\frac{(x_2)_l^*}{(x_2)_l}.
\]

Since $I$ and $J$ are non-admissible, we have
\[
(-1)^{a_1+a_2}
\prod_{l=1}^{a_1}\frac{(x_1)_l^*}{(x_1)_l}
\prod_{l=1}^{a_2}\frac{(x_2)_l^*}{(x_2)_l}
\neq 1,
\]
so the determinant is nonzero. Hence $\operatorname{rank}(A') = 2$.

Now define coefficients along each path by recursive relation in each case.

Let $e_{i,r}$ denote the oriented edge in position $r$ of the path $P_i$. Define the weighted element
\[
\omega
=
\sum_{i=1}^3 \sum_{r=1}^{a_i} y_{i,r}\, e_{i,r}.
\]

Then equation (1) ensures cancellation at all interior vertices of each path, while $A'y'=0$ ensures cancellation at the shared vertices that has degree $3$. Hence
\[
\partial_1(\omega)=0.
\]

Since $y' \neq 0$, we have $\omega \neq 0$. Therefore $\omega$ is a nonzero weighted element supported on $I \vee_C J$ or $I \vee J$ lying in $\ker(\partial_1)$.
\end{proof}
The following theorem generalizes the spanning tree–based construction of the kernel of $\partial_1$ for $N=2$, as described in \cite{efficientH1} by Dey et al., to the case $N\geq 3$, where additional generators beyond single cycles are required.
\begin{theorem}
Let $G$ be a digraph and $G_u$ its underlying undirected graph. Then the kernel of $\partial_1:\Omega_1^N \to \Omega_0^N$ for $N\geq 3$ is generated by:
\begin{enumerate}
    \item weighted admissible cycles, and
    \item weighted elements supported on merges of pairs of non-admissible cycles within the same connected component.
\end{enumerate}
\end{theorem}

\begin{proof}
Let $v \in \ker(\partial_1)$ and let $T$ be a spanning tree of $G_u$. Each non-tree edge $e \in G \setminus T$ determines a fundamental cycle $C_e$. Let
\[
C = \{C_e : e \in G \setminus T\}
\]
be the set of all fundamental cycles.

Partition $C$ into admissible cycles $C_a$ and non-admissible cycles $C_n$.
For each admissible cycle $C_e \in C_a$, by Proposition~\ref{admissiblecycle}, there exists a weighted cycle
\[
c_e = \sum y_i v_i \in \ker(\partial_1)
\]
supported on $C_e$.

If $v$ contains the edge $e$, define $v'=v - c_e$ therefore $v'\in ker(\partial_1)$. Repeating this process, we obtain a vector $v'$ in the kernel that is free from admissible cycles.

If $v'$ is supported only on edges of the spanning tree $T$, then $v'$ cannot lie in the kernel, since the induced subgraph contains a vertex of degree $1$. Hence $v'$ must involve edges corresponding to non-admissible cycles.

Let
\[
C_n = \{C_{e_1}, \dots, C_{e_t}\}
\]
be the non-admissible cycles appearing in $v'$.

For any two cycles $I, J \in C_n$ within the same connected component, consider their merge $I \cup J$. By Proposition~\ref{nonadmissiblecomb}, there exists a nonzero weighted element supported on $I \cup J$ that lies in $\ker(\partial_1)$. Using such elements, we can express $v'$ as a linear combination of kernel elements supported on merges of non-admissible cycles.

Each fundamental cycle $C_e$ contains a unique non-tree edge $e$. For admissible cycles, connecting path for join operation is chosen from the fixed spanning tree where each cycle is connected via unique path.Also for repeated appearance of the same undirected cycle, we select generators supported on merges so that each generator contains a non-tree edge not appearing in any previously chosen generator. 
This is possible because each merge involves at least one fundamental cycle with a unique non-tree edge.
Therefore, any linear relation among the generators forces all coefficients to vanish, and the generating set is linearly independent.

We conclude that $\ker(\partial_1)$ is generated by admissible cycles and kernel elements supported on merges of non-admissible cycles.
\end{proof}

The results of this chapter demonstrate that the structure of the first Mayer path homology group $H_1^{N,q}$ for $N\geq 3$ differs fundamentally from the classical case $N=2$. While undirected cycles generate the kernel in the latter setting, we showed that for higher roots of unity this is no longer sufficient. Instead, only certain weighted cycles, characterized by admissibility conditions, lie in $\ker(\partial_1)$, and additional generators arise from weighted elements supported on merges of non-admissible cycles. In particular, the dimension of $Z_1^{N,1}$ is bounded above by $|E|-|V|+c$ where $c$ is the number connected components of digraph.

\section{Conclusion}
In this work, we introduced Mayer path homology, a new homological framework that unifies Mayer homology and path homology to study directed graphs in a canonical and combinatorially intrinsic manner. By incorporating an $N$-nilpotent boundary operator into the setting of path complexes, this theory overcomes the triangulation dependence inherent in classical Mayer homology while retaining sensitivity to directionality and higher-order interactions. We provided a detailed structural analysis of $\partial$-invariant paths, including explicit classifications of generators in low dimensions, and demonstrated how Mayer path homology refines standard path homology by distinguishing directed network motifs that are otherwise indistinguishable. In particular, our characterization of $N$-chain path complexes and the first Mayer path homology group reveals new algebraic and combinatorial phenomena arising from higher-order differentials. These results establish a   foundation for further theoretical development, such as persistent Mayer path homology and multiscale Mayer path spectral theory, i.e., Mayer path Laplacian and/or Dirac. The potential applications of the present work include topological data analysis, directed networks, and mathematical artificial intelligence. 

\section*{Acknowledgments}
 We thank Vasileios Maroulas for the helpful discussions and insightful comments, which have greatly improved the presentation of this manuscript.
 
\section*{AI Usage Statement}
An AI-based language model was used solely for improving the clarity, grammar, and presentation of the manuscript. The author takes full responsibility for all mathematical results, proofs, and interpretations.

\bibliographystyle{plain}  
\bibliography{references}  

\end{document}